\documentclass[11pt]{article}
\usepackage[centertags]{amsmath}
\usepackage{amsfonts}
\usepackage{amssymb}
\usepackage{makeidx}
\usepackage{newlfont}
\usepackage{amsthm} 
\usepackage{stmaryrd}
\usepackage[]{titletoc}  
\usepackage{mathrsfs}
\usepackage{stmaryrd}
 \usepackage{graphicx}  
 \usepackage{xypic}
\usepackage{hyperref}

\newlength{\defbaselineskip}
\newcommand{\setlinespacing}[1]%
           {\setlength{\baselineskip}{#1 \defbaselineskip}}

\theoremstyle{plain}
\newtheorem{thm}{Theorem}[section]
\newtheorem{cor}[thm]{Corollary}
\newtheorem{lem}[thm]{Lemma}
\newtheorem{prop}[thm]{Proposition}
\newtheorem{exam}[thm]{Example}
\newtheorem{rem}[thm]{Remark}

\newtheorem{Qes}[thm]{Question}

\makeatletter\@addtoreset{equation}{section} \makeatother
\begin{document}

\title {The Periodic Dilation Completeness Problem: Cyclic vectors
in the Hardy space over the infinite-dimensional polydisk
}
\author{
Hui Dan \quad Kunyu Guo}
\date{}
 \maketitle \noindent\textbf{Abstract:}  The classical completeness problem raised by Beurling and independently by Wintner asks for which $\psi\in L^2(0,1)$, the  dilation system $\{\psi(kx):k=1,2,\cdots\}$ is complete in $L^2(0,1)$, where $\psi$ is identified with its extension to an odd $2$-periodic function  on $\mathbb{R}$. This difficult problem is nowadays commonly called as the Periodic Dilation Completeness Problem (PDCP).
By Beurling's idea and an application of the Bohr transform, the  PDCP is translated  as  an equivalent problem  of characterizing cyclic vectors in the
Hardy space $\mathbf{H}_\infty^2$ over the infinite-dimensional polydisk for coordinate multiplication operators.
In this paper, we obtain lots of  new results on cyclic vectors
in the Hardy space $\mathbf{H}_\infty^2$.  In almost all  interesting cases, we obtain sufficient and necessary criterions for characterizing cyclic vectors, and hence in these cases we completely solve the PDCP.  Our results cover almost all previous known results on this subject.



\maketitle

\tableofcontents

\section{Introduction}
In his 1945 seminar, Beurling raised a question on the completeness
of  dilation systems in the Lebesgue space $L^2(0,1)$ (see \cite{Beu} for the note from this seminar). To be precise, extend a function $\psi\in L^2(0,1)$ to an odd periodic function of period 2 defined on the line $\mathbb{R}$. Beurling considered the integer dilation system $\{\psi(kx):k=1,2,\cdots\}$ and sought for conditions to make
$\mathrm{span}\{\psi(kx):k=1,2,\cdots\}$  dense in $L^2(0,1)$. In this case
the system  is referred to as  being \textit{complete}. To study this periodic dilation completeness problem (PDCP for short), Beurling associated the Fourier expansion of $\psi$  with a Dirichlet series and further a  power series in infinitely many variables.
The editors of Beurling's collected works \cite{Beu}  remarked, at the end of the note,  that this problem is actually  an infinite-dimensional version of the invariant subspace problem.

Independently, Wintner considered the PDCP  in 1944 \cite{Win}. He pointed out that the completeness of dilation systems  in $L^2(0,1)$  has
an arithmetical significance due to its connections with various problems in analytic number theory. An important example is the complete sequence
$$1,\ \left\{\frac{x}{2}\right\},\ \left\{2\frac{x}{2}\right\}, \left\{3\frac{x}{2}\right\}, \cdots, \left\{k\frac{x}{2}\right\},\cdots$$
in $L^2(0,1)$, where $\left\{x\right\}$ denotes the fraction part of $x$, the fundamental function of Diophantine approximation theory. This sequence can be interpreted as an
analytic version of the sieve of Eratosthenes.
Also, a series of  deep results on the PDCP can be found in \cite{BM,Bou1,Bou2,Har,HW,Koz1,Koz2,Koz3} around the 1940s,
 some of which involve methods from  Dirichlet series.

 It is  worth mentioning that Nyman showed the equivalence between
 the Riemann hypothesis and completeness of the dilation system $\{\varphi(tx):t>1\}$ in $L^2(0,1)$ in his thesis \cite{Ny}, where $\varphi(x)=\left\{\frac{1}{x}\right\}$. In 2003, B\'{a}ez-Duarte further proved that the Riemann hypothesis holds if and only if the  constant $1$ belongs to the closed linear span of the integer dilation system $\{\varphi(kx):k=1,2,\cdots\}$ in $L^2(0,1)$ (see \cite{BD} and \cite{Ni3}). That is to say, the study of the Riemann hypothesis can be reduced to that of approximation property of this system. Recently, Noor gave the Hardy space $H^2(\mathbb{D})$ version of the B\'{a}ez-Duarte criterion \cite{No}. He constructed a semigroup $\{W_n\}_{n\in\mathbb{N}}$ of weighted composition operators  on $H^2(\mathbb{D})$, which has a connection with power dilation operators $T_n$ by the equality $(I-S)W_n=T_n(I-S)$, and showed the equivalence of the Riemann hypothesis and the existence of  cyclic vectors for $\{W_n\}_{n\in\mathbb{N}}$ in the closure of  the subspace $\mathcal{N}$ constructed in \cite{No}. As shown in \cite{No},  cyclic vectors for $\{W_n\}_{n\in\mathbb{N}}$ have a natural link with PDCP functions. Then all these mentioned above motivate us to focus on the PDCP.

By applying Beurling's idea and  the Bohr transform, the  PDCP is translated  as  an equivalent problem  of characterizing cyclic vectors in the
Hardy space $\mathbf{H}_\infty^2$ over the infinite-dimensional polydisk for coordinate multiplication operators (see Subsection 2.1 for the details). Also as suggested by  Helson, considering power series in
infinitely many variables may be more natural, corresponding to Dirichlet series \cite{Hel}. To be precise, we give a brief statement for some basic notations of   the
Hardy space $\mathbf{H}_\infty^2$.

Let $\mathbb{Z}_+^{(\infty)}$ be the set of all finitely supported sequences of non-negative integers, that is, $\mathbb{Z}_+^{(\infty)}=\bigcup_{n=1}^\infty\mathbb{Z}_+^n$. Then for each $\alpha=(\alpha_1,\alpha_2,\cdots)\in\mathbb{Z}_+^{(\infty)}$, there is a positive integer $N$ such that $\alpha_n=0$ whenever $n>N$. For such $\alpha$, write $\zeta^\alpha$ for the monomial $\zeta_1^{\alpha_1}\cdots\zeta_N^{\alpha_N}$.
The Hardy space $\mathbf{H}_\infty^2$ is defined
to be the Hilbert space consisting of formal power series
$F(\zeta)=\sum_{\alpha\in  \mathbb{Z}_+^{(\infty)}} c_\alpha \zeta^\alpha $
 satisfying
$$\|F\|^2= \sum_{\alpha\in  \mathbb{Z}_+^{(\infty)}} |c_\alpha|^2 <\infty.$$
Set
$$\mathbb{D}_2^\infty=\{\zeta=(\zeta_1,\zeta_2,\cdots)\in l^2:|\zeta_n|<1\  \mathrm{for}\ \mathrm{all}\ n\geq1\}.$$
Then $\mathbb{D}_2^\infty$ is a domain (a  connected open set) in the Hilbert space $l^2$, the space of all square-summable sequences.
By the Cauchy-Schwarz inequality, every series in $\mathbf{H}_\infty^2$ converges absolutely in $\mathbb{D}_2^\infty$ and hence  is a
holomorphic function  on $\mathbb{D}_2^\infty$ (we refer the readers to \cite{Di} for the  definition of holomorphic functions on a domain in some Banach space).
A function $F\in \mathbf{H}_\infty^2$ is said to be \textit{cyclic} if
$[F]=\mathbf{H}_\infty^2$, where
$[F]$ denotes the joint invariant subspace generated by $F$ for the coordinate multiplication operators $M_{\zeta_1}$, $M_{\zeta_2}$, $\cdots$.


As indicated in \cite{HLS},  describing cyclic vectors in  $\mathbf{H}_\infty^2$ is extremely difficult. It is generally known that
cyclic vectors in the  Hardy space $H^2(\mathbb{D}^n)$ over the polydisk $\mathbb{D}^n$ are completely characterized only when $n=1$, and that  the space $H^2(\mathbb{D}^n)$ can be identified with the closed subspace of $\mathbf{H}_\infty^2$ consisting of functions which only depend on the
variables $\zeta_1,\cdots,\zeta_n$. Every cyclic vector in the space $\mathbf{H}_\infty^2$ has necessarily no zeros in $\mathbb{D}_2^\infty$ \cite{Ni1} (indeed this fact also be mentioned  in \cite{Beu} without passing to the space $\mathbf{H}_\infty^2$).
So the PDCP is to find conditions under which functions in $\mathbf{H}_\infty^2$ without zeros are cyclic.
It was shown in \cite{NGN}and independently in \cite{Koz1} that the above necessary condition is also sufficient for  polynomials.
Nikolski generalized this result to the class of functions depending on finitely many
variables $\zeta_1,\cdots,\zeta_n\ (n\in\mathbb{N})$ and holomorphic in some neighborhood of $\overline{\mathbb{D}^n}$ for which  the proof   heavily involves  a growth estimate and a stratification theorem for real analytic manifolds, and is  much more  complex than the polynomials case \cite{Ni1}.

There are some other characterizations  in terms of the Fourier  coefficients of functions in $\mathbf{H}_\infty^2$.
It is well-known that each positive integer $n$ has a unique prime factorization
$n=p_1^{\alpha_1}p_2^{\alpha_2}\cdots p_m^{\alpha_m}$, where $p_j$ is the $j$-th prime number. Thus, one can define a semigroup isomorphism $\alpha$ from  $(\mathbb{N}, \times)$ of  all positive integers onto $(\mathbb{Z}_+^{(\infty)},+)$ by putting $\alpha(n)=(\alpha_1,\alpha_2,\cdots,\alpha_m,0,\cdots).$
So functions in $\mathbf{H}_\infty^2$ can be uniquely expressed  as $\sum_{n=1}^{\infty}a_n\zeta^{\alpha(n)}$ with $\{a_n\}_{n\in\mathbb{N}}$ square-summable.
If a nonzero square-summable sequence $\{a_n\}_{n\in\mathbb{N}}$ is totally multiplicative, that is, $a_{mn}=a_ma_n$ for any $m,n\in\mathbb{N}$,
then the cyclicity of  $F=\sum_{n=1}^{\infty}a_n\zeta^{\alpha(n)}$  was proved in\cite{Har}, and independently stated  in \cite{Koz2}. Some  enlightening proofs also were presented in \cite{HLS} and \cite{Ni1}. When  $F$ is of  multiplicative coefficients, that is, $a_{mn}=a_ma_n$ if $\mathrm{gcd}(m,n)=1$,  Hartman and  Kozlov independently of each other  gave  characterizations for the cyclicity of  $F$  in \cite{Har} and \cite{Koz2}. Using different methods  in \cite{HLS} and \cite{Ni1}, it was shown that  $F=1+\sum_{n=1}^{\infty}a_n\zeta_n\in \mathbf{H}_\infty^2$ is cyclic if and only if $\sum_{n=1}^{\infty}|a_n|\leq1$.

From all previous known results, it seems impossible to give  a universal  criterion for cyclicity in the Hardy space $\mathbf{H}_\infty^2$ as in the Hardy space $H^2(\mathbb{D})$. Hence one is led to classes of  some interesting functions to obtain  criterions  for cyclcity.  In this paper, we will give lots of new  results on cyclic vectors
in the Hardy space $\mathbf{H}_\infty^2$. In almost all interesting cases, we obtain sufficient and necessary criterions for characterizing cyclic vectors, and hence in these cases we  completely solve the PDCP.

Section 2 is dedicated to an introduction to the equivalence  between the PDCP and the problem for characterizing cyclic vectors in $\mathbf{H}_\infty^2$, and some basic function theory for Hardy spaces $\mathbf{H}_\infty^p\ (p>0)$.

The  infinite polydisk algebra $\mathbf{A}_{R,\infty}\ (R>1)$  is introduced  in Section 3.
 For functions in  this  polydisk algebra, we  completely solve the problem of characterizing cyclic vectors. This generalizes   Nikolski's result in finitely many variables \cite{Ni1}.
An immediate  corollary is  that if  the coefficients of $F=\sum_{n=1}^{\infty}a_n\zeta^{\alpha(n)}$ satisfies   one of the following conditions:
1) $\sum_{n=1}^\infty |a_n| n^\varepsilon<\infty$  (for some $\varepsilon>0)$, or 2) $a_n=\hat{f}(n)\ (n\in\mathbb{N})$, the Fourier coefficients of  a holomorphic function $f$ on some neighborhood of the closed unit disk, then $F$ is cyclic if and only if $F$ has no zeros in $\mathbb{D}_2^\infty$.

Section 4 provides a complete characterization for the cyclicity of an  infinite product of functions in $\mathbf{H}_\infty^2$ with mutually independent variables. This result enables us to obtain a criterion for the cyclicity of those functions whose coefficients  have some sort of   multiplicative property. This generalizes  the main result in  \cite{Har} to a more general situation.  In particular, if $F(\zeta)=\sum_{n=1}^\infty (-1)^n a_n\zeta^{\alpha(n)}$ satisfies  $a_{mn}=a_ma_n$ for all $m,n\in\mathbb{N}$, then $F(\zeta)$ is cyclic if and only if  $|a_2|\leq 1/2.$

In Section 5,
we study the cyclicity  from a view of composition operators on $\mathbf{H}_\infty^2$. It is shown  that if $\{\eta_n\}_{n\in\mathbb{N}}$ is a sequence of nonconstant inner functions on $\mathbb{D}_2^\infty$ with mutually independent variables and $\sum_{n=1}^{\infty}|\eta_n(0)|<\infty$, then the composition operator $F\mapsto F(\eta_1,\eta_2,\cdots)$ on $\mathbf{H}_\infty^2$ preserves the cyclicity in the sense that $F(\eta_1,\eta_2,\cdots)$ is cyclic if and only if $F$ is cyclic. For instance, suppose that $\{\eta_n\}_{n\in\mathbb{N}}$ is as above, then $1+\sum_{n=1}^{\infty}a_n\eta_n$ is cyclic if and only if $1+\sum_{n=1}^{\infty}a_n\zeta_n$ is cyclic, if and only if $\sum_{n=1}^{\infty}|a_n|\leq1$.

In Section 6, we proved that if $F\in \mathbf{H}_\infty^p\ (p>2)$ and there exists $G\in \mathbf{H}_\infty^q\ (q>0)$, such that $FG$ is a cyclic vector in $\mathbf{H}_\infty^2$, then $F$ is also cyclic. This generalizes \cite[Theorem 5.7]{HLS} and \cite[Theorem 3.3 (3)]{Ni1}.  As a direct application, it is shown that for each fixed $m\geq 2$, the linear span of the power dilation system $\{\varphi_m(z), \varphi_m(z^2),\cdots\}$ is dense in $H^2(\mathbb{D})$, where $\varphi_m(z)=\log({1-z^m})-\log({1-z})-\log m$.  This is  a stronger version of a central
result in  \cite[Theorem 9]{No}, and hence implies that a weak version of  the Riemann hypothesis  is true.

In section 7, we  use a Riemann surface approach to reveal the relationship between cyclicity of a function $F\in \mathbf{H}_\infty^2$  and the geometry of its image.  It is shown that if  there exists  a simple curve $\gamma$ starting from the origin and tending to the infinity, so that $\gamma$ does not intersect with the image  of $F$, then $F$ is cyclic. In particular, if  the image   of a  function $F$ does not intersect with a  half-straight line starting from  the origin, then $F$ is cyclic. As a simple application of this result, it is shown that if $F\in \mathbf{H}_\infty^2$ is  a nonzero  quasi-homogeneous  function, and $a\in \mathbb{C}$, then  $a+F$ is cyclic if and only if $a+F$ has  not zero point on $\mathbb{D}_2^\infty$ if and only if  $F\in \mathbf{H}^\infty$ and $\|F\|_\infty\leq |a|.$

In the last section, we discuss the   Kozlov completeness problem (see \cite{Koz3, Ni2,Ni3, Ni4}).

\section{Preliminaries}

\subsection{From the PDCP to the cyclicity}

The PDCP  appears in different forms on different spaces, which are  linked together by various unitary  maps. These spaces and maps offer  strategic flexibility for investigating   this problem. Let us explain in detail.

Hedenmalm, Lindqvist and Seip
introduced the Hilbert space of Dirichlet series \cite{HLS} $$\mathcal{H}^2=\{f=\sum_{n=1}^\infty a_nn^{-s}:\|f\|^2=\sum_{n=1}^\infty|a_n|^2<\infty\},$$
and 
studied the PDCP via a unitary operator
\begin{eqnarray*}U:\quad L^2(0,1)& \rightarrow &
  \mathcal{H}^2 ,\\
  \sum_{n=1}^\infty a_n\psi_n& \mapsto & \sum_{n=1}^\infty a_nn^{-s},\end{eqnarray*}
here $\psi_n(x)=\sqrt{2}\sin n\pi x$, and $\{\psi_n\}_{n\in\mathbb{N}}$ is a canonical orthonormal basis for $L^2(0,1)$. They proved that
$\{\psi(kx)\}_{k\in\mathbb{N}}$ is complete in $L^2(0,1)$ if and only if $U\psi$ is
the cyclic vectors in $\mathcal{H}^2$ for  the multiplier algebra of $\mathcal{H}^2$.


The connection between
$\mathcal{H}^2$ and the Hardy space $\mathbf{H}_\infty^2$ was also investigated in their paper.
The unitary transform
\begin{eqnarray*}\mathbf{B}:\mathcal{H}^2& \rightarrow &
  \mathbf{H}_\infty^2 ,\\
  \sum_{n=1}^\infty a_nn^{-s}& \mapsto &\sum_{n=1}^\infty a_n\zeta^{\alpha(n)}\end{eqnarray*}
 is derived from Bohr's brilliant observation \cite{Bo}: a Dirichlet series can be transformed into
 a power series  in infinitely many variables via the``variable" substitution
 $$\zeta_1=p_1^{-s},\zeta_2=p_2^{-s},\cdots,$$
 where  $p_j$ is the $j$-th prime number.
A simple verifying shows that  $\mathbf{B}$ gives a
bijective correspondence between multiplier invariant subspaces of
$\mathcal{H}^2$ and invariant subspaces of $\mathbf{H}_\infty^2$.
In particular, a Dirichlet series $f\in\mathcal{H}^2$ is a cyclic vector in $\mathcal{H}^2$
for the multiplier algebra if and only if $\mathbf{B}f$ is cyclic in $\mathbf{H}_\infty^2$. This builds a connection between the PDCP and cyclicity in $\mathbf{H}_\infty^2$, and thus gives a reasonable explanation to the
remark for Beurling's note \cite{Beu}.

Nikolski exhibited several sufficient conditions for
the PDCP in \cite{Ni1}. All these  were  presented  by using  the terminology of the classic Hardy space $H^2(\mathbb{D})$ over the unit disk $\mathbb{D}$.
Using a unitary operator $V:H_0^2\rightarrow L^2(0,1)$ defined as $Vz^n=\sqrt{2}\sin n\pi x$ $(n\in\mathbb{N})$, he showed that the PDCP is equivalent to characterizing the completeness of power dilation systems in $H_0^2$, here $$H_0^2=\{f\in H^2(\mathbb{D}):\hat{f}(0)=f(0)=0\}.$$
 The Bohr transform $\mathcal{B}$ from  $H_0^2$ onto  $\mathbf{H}_\infty^2$ is defined by putting
$$\mathcal{B}f(\zeta)=\sum_{n=1}^\infty \hat{f}(n)\zeta^{\alpha(n)},\quad f\in H_0^2,\, \zeta\in\mathbb{D}_2^\infty.$$
Then $\mathcal{B}$ is a unitary operator from $H_0^2$ onto $\mathbf{H}_\infty^2$,
and intertwines each power dilation operator $T_k:f\mapsto f(z^k)$ with the multiplication operator $M_{\zeta^{\alpha(k)}}$ on $\mathbf{H}_\infty^2$ in the sense that
for $f\in H_0^2$,
$$
\mathcal{B}T_k f=M_{\zeta^{\alpha(k)}}\mathcal{B} f
,\quad k=1, 2, \cdots.
$$
This gives the correspondence between the completeness of power dilation systems of a function $f$ in $H_0^2$ and cyclicity  of $\mathcal{B} f$ in $\mathbf{H}_\infty^2$; that is,  $\{f(z^k)\}_{k\in\mathbb{N}}$ is complete in $H_0^2$ if and only if $\mathcal{B}f$ is cyclic in $\mathbf{H}_\infty^2$.

We summarize the above statements as a proposition.
\begin{prop}[\cite{HLS, Ni1}] The following conclusions hold.

\begin{itemize}
  \item[(1)]If $\psi\in L^2(0,1)$, then $\{\psi(kt)\}_{k\in\mathbb{N}}$ is complete in $L^2(0,1)$ if and only if $(\mathbf{B}U)\psi$ is cyclic in $\mathbf{H}_\infty^2$.
 \item[(2)]If $f\in H_0^2$, then $\{f(z^k)\}_{k\in\mathbb{N}}$ is complete in $H_0^2$ if and only if $\mathcal{B}f$ is cyclic in $\mathbf{H}_\infty^2$ .
\end{itemize}
\end{prop}

\subsection{Some basic $\mathbf{H}_\infty^p$-function theory}

In this subsection, we  exhibit some basic properties and results on functions in Hardy spaces over $\mathbb{D}_2^\infty$, which will be used later.

Put $$F_{(n)}(\zeta)=F(\zeta_1,\cdots,\zeta_n,0,0,\cdots)
$$
for $F$ holomorphic on $\mathbb{D}_2^\infty$ and $n\in\mathbb{N}$. One can naturally
identify $F_{(n)}$ with a holomorphic function on $\mathbb{D}^n$. The following lemma states a  relation between $F\in \mathbf{H}_\infty^2$ and its restrictions $F_{(n)}$ on $\mathbb{D}^n$ for $n\in \mathbb{N}$.

\begin{lem}\cite{Ni1}\label{l4} If $F$ is holomorphic on $\mathbb{D}_2^\infty$, then $F\in \mathbf{H}_\infty^2$ if and
only if $F_{(n)}\in H^2(\mathbb{D}^n)$ for each $n\in\mathbb{N}$ and $\sup\limits_{n\in\mathbb{N}}\|F_{(n)}\|_{H^2(\mathbb{D}^n)}<\infty$. In this case, $F_{(n)}$ converges to $F$ $(n\rightarrow\infty)$ in the norm of $\mathbf{H}_\infty^2$, and
$$\|F\|_{\mathbf{H}_\infty^2}=\sup\limits_{n\in\mathbb{N}}\|F_{(n)}\|_{H^2(\mathbb{D}^n)}.$$
\end{lem}

Let $\overline{\mathbb{D}}^\infty$
denote the product of countably infinitely many closed unit disks:
$$\overline{\mathbb{D}}^\infty=\overline{\mathbb{D}}\times\overline{\mathbb{D}}\times\cdots.$$
It is known that $\overline{\mathbb{D}}^\infty$ is a compact
Hausdorff space with the product topology.  As usual, write $C(\overline{\mathbb{D}}^\infty)$ for the algebra of all continuous functions on $\overline{\mathbb{D}}^\infty$ with the uniform  norm. Let $\mathbf{A}_{\infty}$ denote the closed subalgebra of $C(\overline{\mathbb{D}}^\infty)$ generated by the coordinate functions $\zeta_1$, $\zeta_2$, $\cdots$. It is  routine to check that the \v{S}hilov boundary  of $\mathbf{A}_{\infty}$ is  the infinite tours
$\mathbb{T}^\infty=\mathbb{T}\times\mathbb{T}\times\cdots.$
Since $\mathbb{T}^\infty$ is a compact Hausdorff group,  it possesses a uniquely normalized  Haar measure $\rho$.  For $0<p<\infty$, define the Hardy space $H^p(\rho)$   to be
the closure of $\mathbf{A}_\infty$
in the space $L^p(\rho)$, and $H^\infty(\rho)$ to be the weak-star closure of $\mathbf{A}_\infty$ in $L^\infty(\rho)$. Then for $1\leq p\leq\infty$, $H^p(\rho)$ is a Banach space; for $0<p<1$, $H^p(\rho)$ is  a topological vector space with the  complete  translation-invariant metric $\mathrm{d}_p(F,G)=\int_{\mathbb{T}^\infty}|F-G|^p\mathrm{d}\rho$.

 For each finite  $p>0$, Cole and Gamelin \cite[Theorem 8.1, also see the discussion after Theorem 6.2]{CG} proved that
the evaluation at each  $\zeta\in\mathbb{D}_2^\infty$ is continuous on $H^p(\rho)$ in the sense that
\begin{equation}\label{|F|p}
 |F(\zeta)|^p\leq\left(\prod_{n=1}^{\infty}
\frac{1}{1-|\zeta_n|^2}\right)\int_{\mathbb{T}^\infty}|F|^p\mathrm{d}\rho,\quad F\in\mathbf{A}_\infty.
\end{equation} Therefore, each function in $H^p(\rho)$ can be uniquely  extended to a holomorphic function on the domain  $\mathbb{D}_2^\infty$ by evaluations such that the above inequality holds.  The Hardy space $\mathbf{H}_\infty^p$  over $\mathbb{D}_2^\infty$  is defined to be the space consisting of all such holomorphic functions, with the inherited metric from $H^p(\rho)$. Then $\mathbf{H}_\infty^p$ is isometrically  isomorphic to
$H^p(\rho)$, and for $p\geq 1$, the isomorphism $H^p(\rho)=\mathbf{H}_\infty^p$
is  realized by taking Poisson integrals \cite{CG}. In the case  $p=\infty$, write $\mathbf{H}^\infty$ for all bounded holomorphic functions  on  $\mathbb{D}_2^\infty$. This is a Banach algebra with  the sup-norm, and is isometrically  isomorphic to
$H^\infty(\rho)$ via the Poisson integral. In fact, as done in \cite{CG}, each function in $\mathbf{H}^\infty$ can be further  analytically extended to  the open unit ball of the sequence space $\textit{c}_0$.  Also, for convenience, we identity a function in  $\mathbf{H}_\infty^p$  with its corresponding ``boundary-value" function in $H^p(\rho)$ for all $p$.

From the above definition  for  Hardy spaces, we present the following proposition which is related to  Nikolski's definition in \cite{Ni1}.
\begin{prop}\label{definition of Hp}
Let $F$ be a holomorphic function on $\mathbb{D}_2^\infty$. Then
$F\in\mathbf{H}_\infty^p$ for some $0< p\leq\infty$ if and only if
 $F_{(n)}\in H^p(\mathbb{D}^n)$ for each $n\in\mathbb{N}$ and $\sup_{n\in\mathbb{N}}\int_{\mathbb{T}^n}|F_{(n)}|^p\mathrm{d}m_n<\infty$. In this case,
\begin{equation}\label{2.2} \int_{\mathbb{T}^\infty}|F|^p\mathrm{d}\rho
 =\sup_{n\in\mathbb{N}}\int_{\mathbb{T}^n}|F_{(n)}|^p\mathrm{d}m_n.
 \end{equation}
\end{prop}
\begin{proof}
The  proof borrows ideals from \cite{Ru} and \cite{AOS},
we sketch it as follows.
  For each finite $p>0$, set
  $${\tilde{\mathbf{H}}_\infty^p}=\{F: F\,\mathrm{ is\ holomorphic \ on \, \mathbb{D}_2^\infty,\,  \ and \, \sup_{n\in\mathbb{N}}\int_{\mathbb{T}^n}|F_{(n)}|^p\mathrm{d}m_n<\infty}\}.$$
 Then   with the aid of (\ref{|F|p}), one can  check that the space ${\tilde{\mathbf{H}}_\infty^p}$ is complete with respect to the metric
  \begin{equation*}
  {\mathrm{d}_p}(F,G)= \begin{cases}
                                   \sup_{n\in\mathbb{N}}\int_{\mathbb{T}^n}|(F-G)_{(n)}|^p\mathrm{d}m_n, & 0<p<1 \\
  \sup_{n\in\mathbb{N}}\|(F-G)_{(n)}\|_{H^p(\mathbb{D}^n)}, & 1\leq p<\infty.
                                 \end{cases}
  \end{equation*}
  By \cite[Theorem 3.4.3]{Ru}, we see that $H^p(\mathbb{D}^n)\subseteq\mathbf{H}_\infty^p$ for each $n\in\mathbb{N}$.
  Since
  the identity map $i$ on the set of polynomials actually defines an isometry from dense polynomial  subspace of $\mathbf{H}_\infty^p$ into ${\tilde{\mathbf{H}}_\infty^p}$,
   the map $i$ can be uniquely extended to an isometrical imbedding from
  $\mathbf{H}_\infty^p$ into ${\tilde{\mathbf{H}}_\infty^p}$. This also implies that the norm equality (2.2) holds.
  So it remains to prove
  that the set of polynomials is also dense in ${\tilde{\mathbf{H}}_\infty^p}$. Since  the set of polynomials  is dense in $\mathbf{A}_\infty$ in the uniform norm, it is sufficient to  show that $\mathbf{A}_\infty$  is  dense in ${\tilde{\mathbf{H}}_\infty^p}$.
 For this end,   taking  $F\in{\tilde{\mathbf{H}}_\infty^p}$, note that the product topology of $\overline{\mathbb{D}}^\infty$ and the topology of $l^2$ coincide on the set
$$r\overline{\mathbb{D}}\times r^2\overline{\mathbb{D}}\times\cdots,\quad 0<r<1,$$
which shows that $F_{\mathrm{r}}(\zeta)=F(r\zeta_1,r^2\zeta_2,\cdots)\in\mathbf{A}_\infty$ because this set is compact in  the topology of $l^2$. Putting $$f_\zeta(z)=F(z\zeta_1,z^2\zeta_2,\cdots),\quad z\in \mathbb{D},\,\zeta\in\mathbb{T}^\infty.$$
then each $f_\zeta$ is holomorphic on $\mathbb{D}$.
 Applying Fubini's theorem the translation-invariance of  Haar measure $\rho$ shows that
\begin{equation}\label{2.3}
 \int_{\mathbb{T}^\infty}\left(\int_{\mathbb{T}}|f_{\zeta}(rz)|^p\mathrm{d}m_1(z)\right)\mathrm{d}\rho(\zeta)
=\int_{\mathbb{T}^\infty}|F(r\zeta_1,r^2\zeta_2,\cdots)|^p\mathrm{d}\rho(\zeta).
\end{equation}
Combining the fact $F_{\mathrm{r}}\in\mathbf{A}_\infty$ and  the  subharmonicity of $ |F_{(n)}|^p$, we see that
\begin{equation}\label{2.4}
\begin{split}
\int_{\mathbb{T}^\infty}|F_{\mathrm{r}}(\zeta)|^p\mathrm{d}\rho(\zeta)&=\sup_n \int_{\mathbb{T}^n}|\big(F_{\mathrm{r}}\big)_{(n)}(\zeta)|^p\mathrm{d}m_n\\
&\leq \sup_n \int_{\mathbb{T}^n}|F_{(n)}(\zeta)|^p\mathrm{d}m_n.
\end{split}
\end{equation}
This shows that  the left side of (2.4) is a  increasing  and bounded function of $r\in(0,1)$. For each $\zeta\in\mathbb{T}^\infty,$  write $$I(\zeta,r)=\int_{\mathbb{T}}|f_{\zeta}(rz)|^p\mathrm{d}m_1(z),\quad\,0<r<1.$$ The subharmonicity of $|f_{\zeta}|^p$ gives that $I(\zeta,r)$ is increasing in $r$.
Then by the monotone convergence theorem and  (2.3),  and (2.4), we have
\begin{equation}\label{2.5}
\int_{\mathbb{T}^\infty}I(\zeta)\mathrm{d}\rho(\zeta)=\lim_{r\rightarrow1^-}\int_{\mathbb{T}^\infty}I(\zeta,r)\mathrm{d}\rho(\zeta)
<\infty,
\end{equation}
where $I(\zeta)=\lim_{r\rightarrow1^-}I(\zeta,r)$.
Therefore, for almost all $\zeta\in\mathbb{T}^\infty$ (with respect to $\rho$),
$I(\zeta,r)$ is bounded in $r$, and thus $f_{\zeta}\in H^p(\mathbb{D})$. In particular, such
$f_{\zeta}$ have radial limits a.e. on $\mathbb{T}$. As done in the proof of \cite[Theorem 3.3.3]{Ru},  by Fubini's theorem, the  limit $\lim_{r\rightarrow1^-}F(r\zeta_1,r^2\zeta_2,\cdots)$,
denoted by $\tilde{F}(\zeta)$, exists a.e. on $\mathbb{T}^\infty$.
This deduces   that $f_\zeta(z)=F(z\zeta_1,z^2\zeta_2,\cdots)$
for almost all $z\in\mathbb{T}$ and $\zeta\in\mathbb{T}^\infty$, and
\begin{equation*}
  \begin{split}
     \int_{\mathbb{T}^\infty}|\tilde{F}(\zeta)|^p\mathrm{d}\rho(\zeta) & =\int_{\mathbb{T}^\infty}\left(\int_{\mathbb{T}}|f_{\zeta}(z)|^p\mathrm{d}m_1(z)\right)\mathrm{d}\rho(\zeta)  \\
     & =\int_{\mathbb{T}^\infty}I(\zeta)\mathrm{d}\rho(\zeta) \\
       & =\sup_{0<r<1}\int_{\mathbb{T}^\infty}|F_{\mathrm{r}}(\zeta)|^p\mathrm{d}\rho(\zeta)\\
        & =\sup_{0<r<1}\sup_n \int_{\mathbb{T}^\infty}|\big(F_{\mathrm{r}}\big)_{(n)}(\zeta)|^p\mathrm{d}m_n\\
         & =\sup_n \int_{\mathbb{T}^\infty}|F_{(n)}(\zeta)|^p\mathrm{d}m_n\\
         &<+\infty,
  \end{split}
\end{equation*}
here the first equality follows from Fubini's theorem, and the third from (2.5) and Fubini's theorem. It follows that $\tilde{F}\in L^p(\rho).$
Applying a similar argument  as in the proof of \cite[Theorem 3.4.3]{Ru} shows that
\begin{equation}\label{2.6}
\int_{\mathbb{T}^\infty}|F_{\mathrm{r}}(\zeta)-\tilde{F}(\zeta)|^p\mathrm{d}\rho(\zeta)\rightarrow0,\quad
r\rightarrow1^-,\end{equation}
and hence $\tilde{F}\in {H}^p(\rho).$  From the   statement following Lemma 2.2,  $\tilde{F}$ can be uniquely  extended to a holomorphic function on the domain  $\mathbb{D}_2^\infty$, still denoted by $\tilde{F}$, and then  $\tilde{F}\in \mathbf{H}_\infty^p$.
By the next reasoning
\begin{equation*}
  \begin{split}
&\sup_n \int_{\mathbb{T}^n}|\big(F_{\mathrm{r}}\big)_{(n)}(\zeta)-F_{(n)}(\zeta)|^p\mathrm{d}m_n\\
 &=\sup_n \int_{\mathbb{T}^n}|\big(F_{\mathrm{r}}\big)_{(n)}(\zeta)-\tilde{F}_{(n)}(\zeta)|^p\mathrm{d}m_n\\
&=\int_{\mathbb{T}^\infty}|F_{\mathrm{r}}(\zeta)-\tilde{F}(\zeta)|^p\mathrm{d}\rho(\zeta)
\end{split}
\end{equation*}
Combining the above reasoning with  (2.6)  shows   ${\mathrm{d}_p}(F_{\mathrm{r}}, F) \rightarrow0,$ as $r\rightarrow1^-.$  This means that  $\mathbf{A}_\infty$  is  dense in
  ${\tilde{\mathbf{H}}_\infty^p}$, and $F=\tilde{F},$  completing the proof in  the case $0<p<\infty$.

 The proof of the case $p=\infty$ is completed by combining some simple real analysis and   Lemma \ref{polynomials w* dense in Hinfty} below.
\end{proof}

Recall that for $F\in \mathbf{H}_\infty^2$, $[F]$ denotes the joint invariant subspace of  $\mathbf{H}_\infty^2$  generated by $F$ for the coordinate multiplication operators.

\begin{lem}[\cite{Ni1}]  \label{polynomials w* dense in Hinfty}  Let $\phi$ be a bounded holomorphic function  on $\mathbb{D}_2^\infty$. Then there exists a sequence of polynomials converging pointwise boundedly to $\phi$. In particular,
if $F\in \mathbf{H}_\infty^2$, then $\phi F\in[F]$.
\end{lem}



   The space $\mathbf{H}_\infty^2$ is a reproducing kernel Hilbert space on $\mathbb{D}_2^{\infty}$
     with the kernels $\mathbf{K}_\lambda(\zeta)=\sum_{\alpha\in\mathbb{Z}_+^{(\infty)}}\overline{\lambda^\alpha}\zeta^\alpha$ ($\lambda\in\mathbb{D}_2^{\infty}$).
We  mention that, the reproducing kernels $\mathbf{K}_\lambda$ of $\mathbf{H}_\infty^2$, distinct from the finite polydisk case, are not always bounded on $\mathbb{D}_2^\infty$. Indeed,

\begin{lem}[\cite{Ni1}]  \label{Klambda is bounded} $\mathbf{K}_\lambda\in \mathbf{H}^\infty$ if and only if $\sum_{n=1}^{\infty}|\lambda_n|<\infty$.
\end{lem}

Now we define, in a natural way, inner functions on $\mathbb{D}_2^\infty$ to be  bounded holomorphic functions on $\mathbb{D}_2^\infty$ with boundary-value function of modulus $1$ almost everywhere on $\mathbb{T}^\infty$.
It is clear that  $\eta$ is an inner function if and only if  the multiplication operator $M_\eta$ with symbol $\eta$
is an isometry on $\mathbf{H}_\infty^2$. Inner functions play an important role in the $\mathbf{H}^p_\infty$-theory. Here we mention three papers by  Bourgin \cite{Bou1, Bou2},   Bourgin and Mendel \cite{BM}. In these papers they  treated the characterization of functions in $L^2(-\pi,\pi)$ which are odd, of period $2\pi$, and such that $\{f(nx)\}$ is orthonormal, and some   completeness problems associated with such $f$. By passing to  the space $L^2(0,1)$, the problem is to consider those $f$ with $\|f\|=1$ such that
$$\langle f(mx), f(nx)\rangle=\delta_{m \,n},\quad m,n=1,2,\cdots.$$
It is easy to verify that the above orthogonality is equivalent to that  $\mathbf{B}U (f)$ is an inner function on $\mathbb{D}_2^\infty$.  Applying the Bohr transform and $\mathbf{H}_\infty^2$-theory, some proofs of results in \cite{Bou1, Bou2, BM} can be greatly  simplified.  For example, combined with \cite[Theorem 3.2]{DGH}, one easily deduces
\cite[Theorems 2.4, 2.5, 2.6, 7.1, 7.2]{BM} and some related  theorems.

We note  that there exists a natural connection between inner functions and isometric dilation theory for doubly
commuting sequences of $C_{.0}$-contractions \cite{DG}. It is also worth mentioning that the preimage of an inner function on $\mathbb{D}_2^\infty$ under the Bohr transform $\mathbf{B}$ is exactly a $\mathcal{H}^2$-inner function, which was introduced in \cite{O}.

\section{Cyclic vectors in the infinite polydisk algebra}

Recall that  $\overline{\mathbb{D}}^\infty$
denotes the product of countably infinitely many closed unit disks.
It is known that $R\,\overline{\mathbb{D}}^\infty$ with the product topology is a compact
metric space for any $R>0$, where the metric $\mathrm{d}(\zeta,\xi)$ is given by
$$\mathrm{d}(\zeta,\xi)=\sum_{n=1}^{\infty}\frac{1}{2^n}\frac{|\zeta_n-\xi_n|}{1+|\zeta_n-\xi_n|},
\quad\zeta,\xi\in R\,\overline{\mathbb{D}}^\infty.$$
Put
$z\zeta=(z\zeta_1,z\zeta_2,\cdots)$ for $\zeta\in\mathbb{C}^\infty$ and $z\in\mathbb{C}$.
It is easy to verify that
\begin{equation}\label{dist}
  \mathrm{d}(z\zeta,z\xi)\leq (R+1)\mathrm{d}(\zeta,\xi)
\end{equation} for each  $\zeta$, $\xi\in \overline{\mathbb{D}}^\infty$ and $z\in R\,\mathbb{D}$.
Define $\mathbf{A}_{R,\infty}$ to be the closed subalgebra of $C(R\,\overline{\mathbb{D}}^\infty)$ generated by the coordinate functions $\zeta_1$, $\zeta_2$, $\cdots$.

In this section, we completely determine when functions in $\mathbf{A}_{R,\infty}\ (R>1)$ are cyclic.

\begin{thm}\label{t5} If $R>1$ and $F\in \mathbf{A}_{R,\infty}$, then $F$ is cyclic if and only if $F$ has no zero in $\mathbb{D}^\infty_2$.
\end{thm}

As an immediate application of Theorem \ref{t5}, we obtain a result in \cite{Ni1}.

\begin{cor} If $F$ is holomorphic on some neighborhood of $\overline{\mathbb{D}}^n\ (n\in\mathbb{N})$, then $F$ is cyclic in $H^2(\mathbb{D}^n)$ for the coordinate multiplication operators $M_{\zeta_1},\cdots,M_{\zeta_n}$ if and only if $F$ has no zero in $\mathbb{D}^n$.
\end{cor}

To prove Theorem \ref{t5}, we need some preparations.
Write $h_r(z)=h(rz)$ for $h\in H^2(\mathbb{D})$ and $0<r<1$. The next lemma can be  derived from the proof of Theorem 1 in\cite{Ros}.   Here we give a straightforward  proof for our purpose.

\begin{lem}\label{l2} Let $f$ be a nonzero holomorphic function on $R\,\mathbb{D}$ for some
$R>1$, and $1<R_1<R_2<R$. If $f$ has no zero in the closed annulus $\{z\in\mathbb{C}:R_1\leq|z|\leq R_2\}$, then for any $h\in H^2(\mathbb{D})$ and $0<r<1$,
$$\|fh_r\|_{H^2(\mathbb{D})}
  \leq \frac{2^d(R_1+R_2)^d\,M}{(R_2-R_1)^d\,m}\|fh\|_{H^2(\mathbb{D})},
$$ where
$M=\max\limits_{|z|=R_2}|f(z)|$,
$m=\min\limits_{|z|=R_2}|f(z)|$,
and $d$ denotes the number of zeros of $f$ in $R_1\mathbb{D}$, counting multiplicity.
\end{lem}
\begin{proof}Let $z_1$, $\cdots$, $z_d$ be the zeros of $f$ in $R_1\mathbb{D}$, and put $$p(z)=(z-z_1)\cdots(z-z_d).$$ Then
$$(R_2-R_1)^d\leq|p|\leq(R_1+R_2)^d$$
on the circle $\{|z|=R_2\}$.
There exists a holomorphic function $g$ on $R\,\mathbb{D}$, such that
$f=pg$ and $g$ has no zero in the closed disk
$R_2\overline{\mathbb{D}}$. For each $z\in\mathbb{T}$, by the maximum modulus principle, $$|g(z)|\leq\max\limits_{|z|=R_2}|g(z)|\leq\frac{M}{(R_2-R_1)^d},$$
and
$$|g(z)|\geq\min\limits_{|z|=R_2}|g(z)|\geq\frac{m}{(R_1+R_2)^d},$$
Let $\tilde{p}$ be the outer factor of $p$ in $H^2(\mathbb{D})$. Note that
$$|z-a|\leq|z-rz|+|rz-a|\leq1-r+|rz-a|\leq2|rz-a|$$
for any $z\in\mathbb{T}$, $|a|\geq1$ and $0<r<1$ (this classical inequality comes from \cite{NGN}, also see \cite{Ge}).
Then $|\tilde{p}|\leq2^d|\tilde{p}_r|$ on $\mathbb{T}$ for $0<r<1$, and
we have
\begin{align} \label{} \int_{\mathbb{T}}|p(z)h(rz)|^2\mathrm{d}m_1(z)
  &=\int_{\mathbb{T}}|\tilde{p}(z)h(rz)|^2\mathrm{d}m_1(z)\notag\\
  &\leq4^d\int_{\mathbb{T}}|\tilde{p}(rz)h(rz)|^2\mathrm{d}m_1(z)\notag\\
  &\leq4^d\int_{\mathbb{T}}|\tilde{p}(z)h(z)|^2\mathrm{d}m_1(z)\notag\\
  &=4^d\int_{\mathbb{T}}|p(z)h(z)|^2\mathrm{d}m_1(z),\label{pol}
\end{align}
here the second inequality follows from the subharmonicity of $|\tilde{p}h|$.
Therefore, by (\ref{pol})
\begin{equation*}\begin{split}\|fh_r\|_{H^2(\mathbb{D})}^2
  &=\int_{\mathbb{T}}|p(z)g(z)h(rz)|^2\mathrm{d}m_1(z)\\
  &\leq\left(\frac{M}{(R_2-R_1)^d}\right)^2\int_{\mathbb{T}}|p(z)h(rz)|^2\mathrm{d}m_1(z)\\
  &\leq\left(\frac{2^d\,M}{(R_2-R_1)^d}\right)^2\int_{\mathbb{T}}|p(z)h(z)|^2\mathrm{d}m_1(z)\\
  &\leq\left(\frac{2^d(R_1+R_2)^d\,M}{(R_2-R_1)^d\,m}\right)^2\int_{\mathbb{T}}|p(z)g(z)h(z)|^2\mathrm{d}m_1(z)\\
  &=\left(\frac{2^d(R_1+R_2)^d\,M}{(R_2-R_1)^d\,m}\right)^2\|fh\|_{H^2(\mathbb{D})}^2,\end{split}
\end{equation*} which completes the proof.
\end{proof}

By Lemma \ref{l2}, for any nonzero holomorphic function $f$  on a neighborhood of $\overline{\mathbb{D}}$, we have $$\mu(f)=\sup_{\substack{h\in H^2(\mathbb{D})\setminus\{0\} \\ 0<r<1 }}\frac{\|fh_r\|_{H^2(\mathbb{D})}}{\|fh\|_{H^2(\mathbb{D})}}<\infty.$$
Furthermore, we show that

\begin{cor}\label{c3} For each $R>1$, $\mu$ is locally bounded on $H^\infty(R\,\mathbb{D})\setminus\{0\}$.
\end{cor}
\begin{proof} Take any nonzero function $f$ in $H^\infty(R\,\mathbb{D})$, we want to show that there exists a positive number $\varepsilon$ and a constant $C>0$, such that
for any $h\in H^2(\mathbb{D})$ and $0<r<1$,
$\|gh_r\|_{H^2(\mathbb{D})}
  \leq C\|gh\|_{H^2(\mathbb{D})}
$
whenever $g\in H^\infty(R\,\mathbb{D})$ and $\|g-f\|_{H^\infty(R\,\mathbb{D})}<\varepsilon$.

Choose $1<R_1<R_2<R$ such that $f$ has no zero in the closed annulus $\{z\in\mathbb{C}:R_1\leq|z|\leq R_2\}$. Put $$m=\min\limits_{R_1\leq|z|\leq R_2}|f(z)|,$$ and let $d$ denote the number of zeros of $f$ in $R_1\mathbb{D}$, counting multiplicity. By Rouch\'{e}'s theorem, if $g\in H^\infty(R\,\mathbb{D})$ and $\|g-f\|_{H^\infty(R\,\mathbb{D})}<\frac m2$, then $g$ has no zero in $\{z\in\mathbb{C}:R_1\leq|z|\leq R_2\}$ and $g$ has $d$ zeros in $R_1\mathbb{D}$.
Set $$C=\frac{2^{d+1}(R_1+R_2)^d(\|f\|_{H^\infty(R\,\mathbb{D})}+\frac m2)}{(R_2-R_1)^d\,m}.$$
 By Lemma \ref{l2}, for such $g$ and any $h\in H^2(\mathbb{D})$, $0<r<1$,
 $$\|gh_r\|_{H^2(\mathbb{D})}
  \leq C\|gh\|_{H^2(\mathbb{D})}.
$$
The proof is complete.
\end{proof}
\vskip2mm


%
%

Put  
$G_r(\zeta)=F(r\zeta)$
for $G$ holomorphic on $\mathbb{D}_2^\infty$ and $0<r<1$.

\begin{lem}\label{p1} Let $F$ be a function in $\mathbf{A}_{R,\infty}$ for some $R>1$ and $F(0)\neq0$. If $G$ is a holomorphic function on $\mathbb{D}_2^\infty$ such that
$FG\in \mathbf{H}_\infty^2$, then for each $0<r<1$, $FG_r\in \mathbf{H}_\infty^2$, and
$FG_r\stackrel{\mathrm{w}}{\rightarrow}FG\ (r\rightarrow1^-)$ in $\mathbf{H}_\infty^2$.
\end{lem}
\begin{proof} Put $$f_\zeta(z)=F(z\zeta),\quad z\in R\,\mathbb{D},\,\zeta\in\overline{\mathbb{D}}^\infty.$$
It is clear that each $f_\zeta$ is a nonzero bounded holomorphic function on $R\,\mathbb{D}$.
Recall that $R\,\overline{\mathbb{D}}^\infty$
is a compact metric space with the metric $\mathrm{d}(\zeta,\xi)$ defined previously.
By (\ref{dist}) and the uniformly continuity of $F$ on $R\,\overline{\mathbb{D}}^\infty$, the map
\begin{eqnarray*}\ell:\overline{\mathbb{D}}^\infty& \rightarrow &
  H^\infty(R\,\mathbb{D})\setminus\{0\},\\
  \zeta& \mapsto & 
  f_\zeta\end{eqnarray*}
is continuous.
Therefore $\ell(\overline{\mathbb{D}}^\infty)$ is a compact subset of $H^\infty(R\,\mathbb{D})\setminus\{0\}$. It follows from Corollary \ref{c3} that $\mu$ is bounded on $\ell(\overline{\mathbb{D}}^\infty)$.
That is to say, there is a constant $C>0$ such that for any $\zeta\in\overline{\mathbb{D}}^\infty$, $h\in H^2(\mathbb{D})$ and $0<r<1$,
\begin{equation}\label{e3}
\|f_\zeta h_r\|_{H^2(\mathbb{D})}
  \leq C\|f_\zeta h\|_{H^2(\mathbb{D})}.
\end{equation}
Set
$$g_\zeta(z)=G(z\zeta),\quad z\in\overline{\mathbb{D}},\,\zeta\in\mathbb{D}_2^\infty,$$
and for each $n\in\mathbb{N}$ we  identify $\overline{\mathbb{D}}^n$ with the closed subset $\overline{\mathbb{D}}^n\times\{0\}\times\cdots$ of $\overline{\mathbb{D}}^\infty$. If $\zeta\in\mathbb{T}^n$ and $0<r<1$, then
$g_{r\zeta}$ is holomorphic on some neighborhood of $\overline{\mathbb{D}}$.
It follows that  for any $n\in\mathbb{N}$ and $0<r<s<1$,
\begin{equation*}\begin{split}\|F_{(n)s}G_{(n)r}\|_{H^2(\mathbb{D}^n)}^2&=\int_{\mathbb{T}^n}|F(s\zeta)G(r\zeta)|^2\mathrm{d}m_n(\zeta)\\
&=\int_{\mathbb{T}}\left(\int_{\mathbb{T}^n}|F(sz\zeta)G(rz\zeta)|^2\mathrm{d}m_n(\zeta)\right)\mathrm{d}m_1(z)\\
&=\int_{\mathbb{T}^n}\left(\int_{\mathbb{T}}|f_{s\zeta}(z)(g_{s\zeta})_{\frac rs}(z)|^2\mathrm{d}m_1(z)\right)\mathrm{d}m_n(\zeta)\\
&\leq C^2\int_{\mathbb{T}^n}\left(\int_{\mathbb{T}}|f_{s\zeta}(z)g_{s\zeta}(z)|^2\mathrm{d}m_1(z)\right)\mathrm{d}m_n(\zeta)\\
&=C^2\int_{\mathbb{T}}\left(\int_{\mathbb{T}^n}|F(sz\zeta)G(sz\zeta)|^2\mathrm{d}m_n(\zeta)\right)\mathrm{d}m_1(z)\\
&=C^2\int_{\mathbb{T}^n}|F(s\zeta)G(s\zeta)|^2\mathrm{d}m_n(\zeta)\\
&\leq C^2\|FG\|_{\mathbf{H}_\infty^2}^2.\end{split}
\end{equation*}
Here the first inequality follows from (\ref{e3}), the last inequality follows from \cite[pp. 1610]{Ni1}.
Therefore $F_{(n)s}G_{(n)r}$ converges weakly to $F_{(n)}G_{(n)r}$ as $s$ approaches $1$ from below, and for any fixed $r$, we have
$$\|(FG_r)_{(n)}\|_{H^2(\mathbb{D}^n)}=\|F_{(n)}G_{(n)r}\|_{H^2(\mathbb{D}^n)}\leq C\|FG\|_{\mathbf{H}_\infty^2}.$$
By Lemma \ref{l4}, $FG_r\in \mathbf{H}_\infty^2$ and $\|FG_r\|_{\mathbf{H}_\infty^2}\leq C\|FG\|_{\mathbf{H}_\infty^2}$.
The rest of the proof is trivial.
\end{proof}

Now we proceed to present the proof of Theorem \ref{t5}.

\noindent  \textbf{Proof of Theorem \ref{t5}.} Suppose that $F$ has no zero in $\mathbb{D}^\infty_2$ and set $G=1/F$ on $\mathbb{D}^\infty_2$. Then $G$ is a holomorphic function on $\mathbb{D}^\infty_2$ and $FG=1$. By Lemma \ref{p1},
for each $0<r<1$, $F/F_r=FG_r\in \mathbf{H}_\infty^2$, and
 $F/F_r\stackrel{\mathrm{w}}{\rightarrow}1\ (r\rightarrow1^-)$ in $\mathbf{H}_\infty^2$.
Now we claim that for any $0<r<1$, $F/F_r$ belongs to
the invariant subspace $[F]$ generated by $F$.
 It suffices to prove that for each $0<r<1$, $F$ has no zero in the compact subset $r\,\overline{\mathbb{D}}^\infty$
of $\overline{\mathbb{D}}^\infty$, and hence $1/F_r$ is bounded on $\mathbb{D}^\infty$. To see this, for any $\zeta\in r\overline{\mathbb{D}}^\infty\setminus\{0\}$, put
$$f(z)=F(z\frac{\zeta}{\|\zeta\|_{l^\infty}}),\quad z\in\mathbb{D}$$ and
$$f_n(z)=F_{(n)}(z\frac{\zeta}{\|\zeta\|_{l^\infty}}),\quad z\in\mathbb{D}$$ for $n\in\mathbb{N}$. Then $\{f_n\}_{n\in\mathbb{N}}$ is uniformly bounded  and  each $f_n$ has no zero
in $\mathbb{D}$. Since $F$ is continuous on $\overline{\mathbb{D}}^\infty$, $\{f_n\}_{n\in\mathbb{N}}$ converges pointwise to $f$ on  $\mathbb{D}$. This implies that $\{f_n\}_{n\in\mathbb{N}}$ converges uniformly to $f$ on compact subsets of $\mathbb{D}$. Note that $f(0)=F(0)\neq0$.
It follows from Hurwitz's theorem that $f$ also
has no zero in $\mathbb{D}$. In particular, $$F(\zeta)=f(\|\zeta\|_{l^\infty})\neq0.$$
We have proved the claim, and thus complete the proof of the theorem.
$\hfill \square $
\vskip2mm

\begin{rem}\label{r1} Theorem \ref{t5} fails for $R=1$. In fact, put $$F(\zeta)=(1-\zeta_1)\exp\left(\frac{\zeta_1+1}{\zeta_1-1}\right).$$
Then $F$ is a function in $\mathbf{A}_\infty$ without zeros in $\mathbb{D}_2^\infty$. On the other hand, since the function $$(1-z)\exp\left(\frac{z+1}{z-1}\right)$$
has a nontrivial inner factor, $F$ cannot be cyclic
in $\mathbf{H}_\infty^2$ (see Lemma \ref{cyclicity in H2(S)}).
\end{rem}


Let us see the following application of Theorem \ref{t5}.
\begin{cor} Suppose $F=\sum_{n=1}^\infty a_n\zeta^{\alpha_{(n)}}\in \mathbf{H}_\infty^2$. If there exists $\varepsilon>0$ such that
the coefficients of $F$ satisfies the growth condition:
$$\sum_{n=1}^\infty |a_n| n^\varepsilon<\infty,$$
then $F\in \mathbf{A}_{R,\infty}$ where $R=2^\varepsilon$, and therefore
$F$ is cyclic
in $\mathbf{H}_\infty^2$ if and only if $F$
has no zero in $\mathbb{D}_2^\infty$. In particular, if $a_n=\hat{f}(n)\ (n\in\mathbb{N})$, the Fourier coefficients of  a holomorphic function $f$ on some neighborhood of the closed unit disk, then $F$ is cyclic
in $\mathbf{H}_\infty^2$ if and only if $F$
has no zero in $\mathbb{D}_2^\infty$.
\end{cor}
\begin{proof}
Put $R=2^\varepsilon>1$, and for any $n\in\mathbb{N}$, let
$n=p_1^{\alpha_1}p_2^{\alpha_2}\cdots p_m^{\alpha_m}$
be its  prime factorization.
Thus for each $\zeta\in R\,\overline{\mathbb{D}}^\infty$,
$$
 |\zeta^{\alpha(n)}|\leq R^{\alpha_1+\cdots+\alpha_m}\leq p_1^{\varepsilon\alpha_1}p_2^{\varepsilon\alpha_2}\cdots p_m^{\varepsilon\alpha_m}=n^\varepsilon.
$$
Therefore for any $\zeta\in R\,\overline{\mathbb{D}}^\infty$, we have
$$|F(\zeta)-\sum_{n=1}^N a_n\zeta^{\alpha(n)}|=\left|\sum_{n=N+1}^\infty a_n\zeta^{\alpha(n)}\right|\leq\sum_{n=N+1}^\infty \left|a_n\right|\, \left|\zeta^{\alpha(n)}\right|\leq\sum_{n=N+1}^\infty |a_n|n^\varepsilon.$$
A simple estimation shows  $F\in \mathbf{A}_{R,\infty}$.
The rest of the proof is easy.
\end{proof}

In what follows let us see how much  cyclic vectors are in the sense of Baire's category.
Let $\mathcal{F}$ denote all nonvanishing functions  in $\mathbf{H}_\infty^2$, together with the zero function,
 that is,  $$\mathcal{F}=\{0\}\cup\{f\in \mathbf{H}_\infty^2: f(\zeta)\not=0, \zeta\in\mathbb{D}_2^\infty\}.$$
We  will  apply Theorem \ref{t5}  and an idea in \cite{Sh} to show that the following proposition is true.
\begin{prop} We have
\noindent\begin{itemize}
\item[(1)]   $ \mathcal{F}$ is a closed subset of $\mathbf{H}_\infty^2$;
 \item[(2)]   the set $\mathcal{C}$ of all cyclic vectors in $\mathbf{H}_\infty^2$ is dense in $\mathcal{F}$, and of the second category in $\mathcal{F}$;
\item[(3)]  $\mathcal{F}\setminus\mathcal{C}$ is of the first category in $\mathcal{F}$.
\end{itemize}
\end{prop}
\begin{proof}
We first note the fact that if $\{F_n\}$ converges to $F$ in the norm of  $\mathbf{H}_\infty^2$, then $\{F_n\}$ converges uniformly to $F$ on any compact subset $K$ of $\mathbb{D}_2^\infty$. In fact, since $\zeta\mapsto\|\mathbf{K}_\zeta\|$ is continuous on $\mathbb{D}_2^\infty$,
$\|\mathbf{K}_\zeta\|$ is uniformly bounded on $K$. Then,  as $n\rightarrow\infty$, we have
$$\sup_{\zeta\in K}|F(\zeta)-F_n(\zeta)|=\sup_{\zeta\in K}|\langle F-F_n,\mathbf{K}_\zeta\rangle|\leq\|F-F_n\|\cdot\sup_{\zeta\in K}\|\mathbf{K}_\zeta\|\rightarrow 0.$$

In what follows that  we will  prove that $\mathcal{F}$ is closed in $\mathbf{H}_\infty^2$. For this, suppose that $\{F_n\}$ is a sequence in $\mathcal{F}\setminus\{0\}$ converging to  some function $F\in\mathbf{H}_\infty^2$. We  will show $F\in\mathcal{F}$.
For a given $\zeta=(\zeta_1,\zeta_2,\cdots)\in\mathbb{D}_2^\infty\setminus\{0\}$, put $r=\max_{n\in\mathbb{N}}\{|\zeta_n|\}$ and define functions on $\mathbb{D}$ by $f(z)=F(\frac{z}{r}\zeta_1,\frac{z}{r}\zeta_2,\cdots)$ and
$f_n(z)=F_n(\frac{z}{r}\zeta_1,\frac{z}{r}\zeta_2,\cdots)\ (n\in\mathbb{N})$. Then $f$ and all $f_n$  are  holomorphic on $\mathbb{D}$. By the above fact,  $f_n$ converges uniformly to $f$ on each compact subset of $\mathbb{D}$. Since each $f_n$ has no zeros in $\mathbb{D}$, it follows from Hurwitz's theorem that $F(0)=f(0)$ and $F(\zeta)=f(r)$ equal to $0$ or not simultaneously. This means that if $F(0)=0$, then $F$ is the zero function; if $F(0)\neq0$, then $F$ has no zeros in $\mathbb{D}_2^\infty$. This shows $F\in\mathcal{F}$.

Let us continue to  prove that $\mathcal{C}$ is dense in $\mathcal{F}$. Obviously, the zero function belongs to the closure of $\mathcal{C}$. Now fix a function $F=\sum_{\alpha\in\mathbb{Z}_+^{(\infty)}} c_\alpha\zeta^{\alpha}\in\mathbf{H}_\infty^2$ without zeros in $\mathbb{D}_2^\infty$, and
set $F_{\mathbf{r}}(\zeta)=F(r\zeta_1,r^2\zeta_2,\cdots)$ for $0<r<1$. It is easy to see that $F_{\mathbf{r}}$ converges to $F$ as $r$ approaches $1$ from below. So it suffices to prove that for any given $0<r<1$, $F_{\mathbf{r}}$ is cyclic. Take $R>1$ satisfying  $rR<1$.
Since $K=rR\,\overline{\mathbb{D}}\times r^2R\,\overline{\mathbb{D}}\times\cdots$
 is compact in $\mathbb{D}_2^\infty$, $\{F_{(n)}\}$ converges uniformly to $F$ on $K$. This means  that $\{\big(F_{\mathbf{r}}\big)_{(n)}\}$ converges uniformly to $F_{\mathbf{r}}$ on $R\,\overline{\mathbb{D}}^\infty$, and hence  $F_{\mathbf{r}}\in\mathbf{A}_{R,\infty}$. Then by Theorem \ref{t5}, $F_{\mathbf{r}}$ is cyclic, completing the proof of the denseness of    $\mathcal{C}$ in $\mathcal{F}$.

Finally,  we put
$$\delta(F)=\inf\{\|1-pF\|:p\ \mathrm{ run\   through \ all\ polynomials}\},\quad F\in\mathbf{H}_\infty^2.$$ It is clear that a function $F\in\mathbf{H}_\infty^2$ is cyclic exactly when $\delta(F)=0$.
We also see that $\delta$ is upper semicontinuous on $\mathbf{H}_\infty^2$. In fact,
if $\{F_n\}$ converges to $F$, then for each polynomial $p$,
$$\delta(F_n)\leq\|1-pF_n\|\rightarrow\|1-pF\|,\quad n\rightarrow\infty,$$
forcing $\|1-pF\|\geq\overline{\lim}_{n\rightarrow\infty}\delta(F_n)$. Therefore, one has $\delta(F)\geq\overline{\lim}_{n\rightarrow\infty}\delta(F_n)$. The upper semicontinuity of $\delta$ together with
the denseness of $\mathcal{C}$  in $\mathcal{F}$
implies that
$\{F\in\mathcal{F}:\delta(F)\geq\frac{1}{n}\}$ is closed and nowhere dense in $\mathcal{F}$, and then $$\mathcal{F}\setminus\mathcal{C}
=\bigcup_{n=1}^\infty\{F\in\mathcal{F}:\delta(F)\geq\frac{1}{n}\}$$ is of the first category in $\mathcal{F}$. Also, by the Baire category theorem, $\mathcal{C}$ is of the second category in $\mathcal{F}$.
\end{proof}

\section{The cyclicity of  product functions}

In this section, we give a complete characterization for the cyclicity of infinite product of functions with mutually independent variables. Furthermore,  these  results will be used to study  the  cyclicity  of functions with  some sort of  multiplicative coefficients.

\subsection{The cyclicity of  product functions}

Following \cite{Ni1}, for a nonempty subset $S\subseteq\mathbb{N}$, let $H^2(\mathbb{D}^S)$ denote
the closed linear span of the set $$\{\zeta^\alpha:\alpha=(\alpha_1,\alpha_2,\cdots)\in\mathbb{Z}_+^{(\infty)},\  \alpha_n=0\ \mathrm{for}\ \mathrm{any}\ n\neq S\}$$ of monomials.
That is to say, functions in $H^2(\mathbb{D}^S)$
 depend only on the variables $\{\zeta_n\}_{n\in S}$. By a sequence of functions $\{F_n\}_{n\in\mathbb{N}}$ in $\mathbf{H}_\infty^2$ with mutually independent variables we mean that  there is a sequence of nonempty subsets $\{S_n\}_{n\in\mathbb{N}}$  of $\mathbb{N}$ which are pairwise disjoint, and $\bigcup_nS_n=\mathbb{N}$ such that each $F_n\in H^2(\mathbb{D}^{S_n})$ for $n\in\mathbb{N}$.

We establish the following convergence criterion for infinite product of functions with mutually independent variables, which will be proved later.

\begin{prop} \label{prod lemma}
 Let $\{F_n\}_{n\in\mathbb{N}}$ be a sequence of functions in $\mathbf{H}_\infty^2$ with mutually independent variables. Then the following are equivalent:
\noindent\begin{itemize}
\item[(1)]  the infinite product $\prod_{n=1}^{\infty}F_n$ converges  to
 a nonzero function $F$  in the norm of $\mathbf{H}_\infty^2$;
 \item[(2)]   the infinite product $\prod_{n=1}^{\infty}F_n$ converges pointwise to
 a nonzero function $F\in \mathbf{H}_\infty^2$;
\item[(3)]  $\prod_{n=1}^{\infty}\|F_n\|$ converges and $\lim_{n\rightarrow\infty}\prod_{m=n+1}^{\infty}F_m(0)=1$.
\end{itemize}
\end{prop}
\vskip1,5mm
Our main result in this subsection completely characterizes  the cyclicity of infinite product of functions with mutually independent variables.

\begin{thm}\label{prod theorem} Suppose that $\{F_n\}_{n\in\mathbb{N}}$ is a sequence of  functions in $\mathbf{H}_\infty^2$ with mutually independent variables and $\prod_{n=1}^{\infty}F_n$ converging pointwise to
 a nonzero function $F\in \mathbf{H}_\infty^2$. Then
  $F$ is cyclic if and only if $F_n$ is cyclic
  for each $n\in\mathbb{N}$.
\end{thm}

We need some preparations for the proof of Theorem \ref{prod theorem}. The next lemma comes from Lemma 3.2 in \cite{Ni1}.

\begin{lem} \label{cyclicity in H2(S)} Let $S$ be a nonempty subset of $\mathbb{N}$ and $F$ be a function in $H^2(\mathbb{D}^S)$. Then
 $F$ is cyclic in
$H^2(\mathbb{D}^S)$ for $\{M_{\zeta_n}\}_{n\in S}$
if and only if $F$ is cyclic in $\mathbf{H}_\infty^2$ .
\end{lem}

We will use  the following fact. Suppose that $F_1,F_2\in \mathbf{H}_\infty^2$ have mutually independent variables, that is, there exist   disjoint nonempty subsets $S_1,S_2\subseteq\mathbb{N}$ so that $F_1\in H^2(\mathbb{D}_2^{S_1})$ and $F_2\in H^2(\mathbb{D}_2^{S_2})$. Then $F_1F_2\in H^2(\mathbb{D}_2^{S_1\cup S_2})$ and $\|F_1F_2\|=\|F_1\|\|F_2\|$.

\vskip2mm
Now we first give the proof of Proposition \ref{prod lemma}.

\noindent\textbf{Proof of Proposition \ref{prod lemma}.} (1)$\Rightarrow$ (2) is obvious. (2)$\Rightarrow$ (3)
Assume first that $\prod_{n=1}^{\infty}F_n$ converges pointwise to
 a nonzero function $F\in \mathbf{H}_\infty^2$. Since $F_1,F_2,\cdots$ have mutually independent variables, there is a sequence $\{S_n\}_{n\in\mathbb{N}}$ of  pairwise disjoint subsets of $\mathbb{N}$, such that $F_n\in H^2(\mathbb{D}^{S_n})$ and $\bigcup_{n=1}^\infty S_n=\mathbb{N}$. By Lemma \ref{l4}, one can take $l\in\mathbb{N}$ sufficiently large so that $F_{(l)}\neq0$, and  $F(\lambda_1,\cdots,\lambda_l,0,0,\cdots)\neq0$ for some $\lambda_1,\cdots,\lambda_l\in\mathbb{D}$. Since $S_1,S_2,\cdots$ are pairwise disjoint, there exists $N\in\mathbb{N}$ such that for each $n>N$, $S_n$ does not contain any numbers in $\{1,\cdots,l\}$. Therefore,
 $$\prod_{n=N+1}^{\infty}F_n(0)=\prod_{n=N+1}^{\infty}F_n(\lambda_1,\cdots,\lambda_l,0,0,\cdots)$$
 converges, which gives $\lim_{n\rightarrow\infty}\prod_{m=n+1}^{\infty}F_m(0)=1$. For the convergence of $\prod_{n=1}^{\infty}\|F_n\|$,
 let $P_n\ (n\in\mathbb{N})$ denote the orthogonal projection from $\mathbf{H}_\infty^2$ onto $H^2(\mathbb{D}^{R_n})$, where $R_n=S_1\cup\cdots\cup S_n$. Since $\bigcup_{n=1}^\infty S_n=\mathbb{N}$, we see that $P_np\rightarrow p$ in the norm of $\mathbf{H}_\infty^2$ (as $n\rightarrow\infty$) for any polynomial $p$, and then $\{P_n\}_{n\in\mathbb{N}}$ converges to the identity operator $I$ in the strong operator topology. It is routine to check that $P_nF=F_1\cdots F_n\cdot \prod_{m=n+1}^{\infty}F_m(0)$, and hence
 $$\|F\|=\lim_{n\rightarrow\infty}\|P_nF\|=\lim_{n\rightarrow\infty}\|F_1\|\dots\|F_n\|\cdot \prod_{m=n+1}^{\infty}|F_m(0)|=\prod_{n=1}^\infty\|F_n\|.$$

(3)$\Rightarrow$ (1).
Now suppose that the infinite product
$\prod_{n=1}^{\infty}\|F_n\|$ converges  and $\lim_{n\rightarrow\infty}\prod_{m=n+1}^{\infty}F_m(0)=1$.
Without loss of generality, we may assume that $F_n(0)=1$ for each $n\in\mathbb{N}$.
Put $G_n=F_1\cdots F_n\ (n\in\mathbb{N})$. Then $\|G_n\|=\|F_1\|\cdots\|F_n\|$ is uniformly bounded for all $n\in\mathbb{N}$. We also see that if $m>n$,
then $$\langle G_m,G_n\rangle=\langle G_n\prod_{l=n+1}^{m}F_l,G_n\rangle=\|G_n\|^2,$$
and therefore $\|G_m-G_n\|^2=\|G_m\|^2-\|G_n\|^2$. This shows that $\{G_n\}_{n\in\mathbb{N}}$ is a Cauchy sequence in $\mathbf{H}_\infty^2$. This together with $\lim_{n\rightarrow\infty}\|G_n\|=\prod_{n=1}^{\infty}\|F_n\|$, implies that
$\{G_n\}_{n\in\mathbb{N}}$ converges  to a nonzero function in $\mathbf{H}_\infty^2$ in the norm of $\mathbf{H}_\infty^2$.
$\hfill \square $
\vskip2mm

We are ready to prove Theorem \ref{prod theorem}.

\noindent\textbf{Proof of Theorem \ref{prod theorem}.}
Take a sequence $\{S_n\}_{n\in\mathbb{N}}$ of  pairwise disjoint subsets of $\mathbb{N}$, such that $F_n\in H^2(\mathbb{D}^{S_n})$.

 Assume that   $F$ is cyclic. One can take a sequence $\{p_k\}_{k\in\mathbb{N}}$ of polynomials in variables $\{\zeta_n:n\notin S_1\}$ that converges to $\prod_{n=2}^{\infty}F_n$  in the norm of $\mathbf{H}_\infty^2$. Therefore, we have
 $$\|F_1p_k-F\|=\|F_1(p_k-\prod_{n=2}^{\infty}F_n)\|=\|F_1\|\|p_k-\prod_{n=2}^{\infty}F_n\|\rightarrow0,\quad k\rightarrow\infty.$$
This gives $F\in[F_1]$, forcing $F_1$ to be cyclic.
The same reasoning shows that  $F_n$ is cyclic
  for each $n\geq2$.

 For the converse, assume that $F_n$ is cyclic
  for each $n\in\mathbb{N}$. Put $G_n=\prod_{m=n}^{\infty}F_m$.
Then by Proposition \ref{prod lemma}, $\|G_n\|
=\prod_{m=n}^\infty \|F_m\|$ is uniformly bounded for all $n\in\mathbb{N}$,
and $\{G_n\}_{n\in\mathbb{N}}$ converges pointwise to the constant function $1$ on $\mathbb{D}_2^\infty\setminus Z(F)$, where
$Z(F)=\{\zeta\in\mathbb{D}_2^\infty:F(\zeta)=0\}$, the zero set of $F$. Since $F\neq0$,
$\mathbb{D}_2^\infty\setminus Z(F)$ is dense in
$\mathbb{D}_2^\infty$ (in the topology inherited
from  $l^2$).
A standard argument shows that $$G_n\stackrel{\mathrm{w}}{\rightarrow}1\quad (n\rightarrow\infty)$$ in $\mathbf{H}_\infty^2$.
To complete the proof, it suffices to prove that each $G_n$ is in
the invariant subspace $[F]$ generated by $F$.
We will show this by induction. Clearly $G_1=F$ is in $[F]$. Suppose that we have shown for some $n\in\mathbb{N}$ that $G_n\in[F]$. Since $F_n$ is cyclic, $F_n$ is cyclic in $H^2(\mathbb{D}_2^{S_n})$ by Lemma \ref{cyclicity in H2(S)}. We can pick a sequence of polynomials $\{q_k\}_{k\in\mathbb{N}}$ in variables $\{\zeta_m:m\in S_n\}$ such that $q_kF_n\to1\ (k\rightarrow\infty)$ in the norm of $\mathbf{H}_\infty^2$.
Therefore, we have
\begin{equation*}
\begin{split}
   \|q_kG_n-G_{n+1}\|
   &=\|q_kF_nG_{n+1}-G_{n+1}\|\\
     &=\|(q_kF_n-1)G_{n+1}\|\\
     &=\|q_kF_n-1\|\|G_{n+1}\|\rightarrow0
\end{split}
\end{equation*} as $k\to\infty$.
This gives that $G_{n+1}\in[G_n]\subseteq[F]$. By induction, the proof is complete.
$\hfill \square $
\vskip2mm

\subsection{ The  cyclicity  of functions with  some sort of multiplicative coefficients}

Let $\{a_n\}_{n\in\mathbb{N}}$ be a sequence of complex numbers. The sequence $\{a_n\}_{n\in\mathbb{N}}$ is said to be
\begin{itemize}
  \item [(1)] totally multiplicative if $a_{mn}=a_ma_n$ for any $m,n\in\mathbb{N}$;
  \item [(2)] multiplicative if $a_{mn}=a_ma_n$ for any pair $m,n$ of coprime positive integers.
\end{itemize}
It was shown in \cite{Har} that if the coefficients of a nonzero $\mathbf{H}_\infty^2$-function $F=\sum_{n=1}^\infty a_n\zeta^{\alpha(n)}$ are totally multiplicative, then $F$ is cyclic.
Some  different proofs also were presented in \cite{HLS}, \cite{Ni1} and \cite{Koz2}.
Below we give a simple proof by virtue of Theorem \ref{prod theorem}.
By the fact that a nonzero square-summable sequence $\{a_n\}_{n\in\mathbb{N}}$ is totally multiplicative if and only if
there exists a point $\lambda\in\mathbb{D}_2^\infty$, such that $a_n=\overline{\lambda^{\alpha(n)}}$ for each $n\in\mathbb{N}$, it follows that those nonzero functions in $\mathbf{H}_\infty^2$ with totally multiplicative coefficients are exactly reproducing kernels $\mathbf{K}_\lambda(\zeta)=\sum_{n=1}^\infty\overline{\lambda^{\alpha(n)}}\zeta^{\alpha(n)}$ of $\mathbf{H}_\infty^2$.
It is also easy to see that each $\mathbf{K}_\lambda$ is a infinite product,
$$\mathbf{K}_\lambda(\zeta)=\sum_{n=1}^\infty\overline{\lambda^{\alpha(n)}}\zeta^{\alpha(n)}
=\sum_{n=1}^\infty(\bar{\lambda}\zeta)^{\alpha(n)}
=\prod_{n=1}^\infty\left(\sum_{m=1}^\infty(\overline{\lambda_n}\zeta_n)^m\right)
=\prod_{n=1}^\infty\frac{1}{1-\overline{\lambda_n}\zeta_n}.$$
Since every component in this infinite product
is an outer function in $H^2(\mathbb{D})$, Theorem \ref{prod theorem} together with Lemma \ref{cyclicity in H2(S)} immediately gives that every reproducing kernel of $\mathbf{H}_\infty^2$ is cyclic.

Using the language of Dirichlet series,
Hartman considered the case  in which  the coefficients  are  multiplicative \cite{Har}. By the Bohr transform, his conclusion can be translated as: let $F=\sum_{n=1}^\infty a_n\zeta^{\alpha(n)}$ be a function in $\mathbf{H}_\infty^2$   with multiplicative coefficients, and put $f_p=\sum_{k=0}^{\infty}a_{p^k}z^k$ for each prime number $p$,
then $F$ is cyclic if and only if each $f_p$ is an outer function in $H^2(\mathbb{D})$.  Kozlov also independently  obtained  the same  characterization for the cyclicity of  $F$ with  multiplicative coefficients in \cite{Koz2}.

Below we  apply Theorem \ref{prod theorem} to generalize these  results to a  more general situation.
Let $\gcd(m,n)$ denote the greatest common divisor of a pair $m,n$ of positive integers, and $\mathcal{P}$ denote the set of all prime numbers. For a subset  $S$  of $\mathcal{P}$, and a positive integer $n$ if $\gcd(n,p)=1$
for each $p\in S$, then we say that $n$ is coprime to $S$. This is equivalent to $\mathbf{P}(n)\cap S=\emptyset$, where $\mathbf{P}(n)$ denotes the set of all prime factors of $n$.
Now let $\mathfrak{\Delta}=\{\mathcal{P}_1,\mathcal{P}_2,\cdots\}$ be a partition of $\mathcal{P}$, that is, $\mathcal{P}_i\cap\mathcal{P}_j=\emptyset$ for $i\not=j$ and $\bigcup_k\mathcal{P}_k=\mathcal{P}.$
Say that two positive integers $m,n$ are \textit{$\mathfrak{\Delta}$-coprime} if for each $k$, either $m$ or $n$ is coprime to  $\mathcal{P}_k$. A sequence $\{a_n\}_{n\in\mathbb{N}}$ of complex numbers is said to be \textit{$\mathfrak{\Delta}$-multiplicative}  if $a_{mn}=a_ma_n$ for any pair $m,n$ of $\mathfrak{\Delta}$-coprime  positive integers.

 It is clear that any positive integer $n$ has a unique ``prime factorization" with respect to some partition $\mathfrak{\Delta}=\{\mathcal{P}_1,\mathcal{P}_2,\cdots\}$ of $\mathcal{P}$. To be more precise,
 let $\mathcal{N}_k\ (k\in\mathbb{N})$ denote the subsemigroup of $\mathbb{N}$ generated by $\mathcal{P}_k\cup\{1\}$ under multiplication. Every  positive integer $n$ can be decomposed uniquely as a product $n=\prod_{k=1}^{\infty}n_k$, where $n_k\in\mathcal{N}_k$ and there exists $N\in\mathbb{N}$ such that $n_k=1$ for any $k>N$. Furthermore, we see that $n_1,\cdots,n_N$ are pairwise $\mathfrak{\Delta}$-coprime.

 A partition {$\mathfrak{\Delta'}$} is called a refinement of  {$\mathfrak{\Delta}$} if each set in {$\mathfrak{\Delta'}$} is contained in some set in  $\mathfrak{\Delta}$. It is easy to verify that if $\{a_n\}_{n\in\mathbb{N}}$ is $\mathfrak{\Delta'}$-multiplicative, then  it also  is $\mathfrak{\Delta}$-multiplicative. The  single-point  partition $\mathfrak{\partial}=\{\{p_1\},\{p_2\},\cdots\}$ of $\mathcal{P}$ is the finest  partition. Then two positive integers $m,n$ are  $\mathfrak{\partial}$-coprime exactly when they  are  coprime in the usual sense. In this case, the notion of $\mathfrak{\partial}$-multiplicative sequences coincides with that of multiplicative sequences.



For  functions with $\mathfrak{\Delta}$-multiplicative coefficients, we   have the following result.

\begin{thm} \label{cyclicity Ssim multiplicative} Let $\mathfrak{\Delta}=\{\mathcal{P}_1,\mathcal{P}_2,\cdots\}$ be a partition of $\mathcal{P}$, and
$F=\sum_{n=1}^\infty a_n\zeta^{\alpha(n)}$ be a function in $\mathbf{H}_\infty^2$   with $\mathfrak{\Delta}$-multiplicative coefficients. For each $k\in\mathbb{N}$, put $F_k=\sum_{m\in\mathcal{N}_k}a_m\zeta^{\alpha(m)}$. Then $F$ is cyclic if and only if each $F_k$ is cyclic.
\end{thm}

To prove Theorem \ref{cyclicity Ssim multiplicative}, we need the following lemma.

\begin{lem} \label{summable lem} Let $\mathfrak{\Delta}=\{\mathcal{P}_1,\mathcal{P}_2,\cdots\}$ be a partition of $\mathcal{P}$.
  If an absolute summable sequence $\{a_n\}_{n\in\mathbb{N}}$ is $\mathfrak{\Delta}$-multiplicative,
  then $$\sum_{n=1}^{\infty}a_n=\prod_{k=1}^\infty\sum_{m\in\mathcal{N}_k}a_m.$$

\end{lem}
\begin{proof}
We only prove the case that the carnality of $\mathfrak{\Delta}$ is infinite. When the carnality is finite, the proof is more easier.  For $k\in\mathbb{N}$, put
$$\mathcal{N}_1\cdots\mathcal{N}_k=\{n_1\cdots n_k:n_i\in\mathcal{N}_i, i=1,\cdots,k\}.$$
 Since $\mathbb{N}=\bigcup_{k=1}^{\infty}\mathcal{N}_1\cdots\mathcal{N}_k$, we have $$\sum_{n=1}^{\infty}a_n=\lim_{k\rightarrow\infty}
 \sum_{n\in\mathcal{N}_1\cdots\mathcal{N}_k}a_n.$$ It suffices to prove that for any given
$k\in\mathbb{N}$,
\begin{equation}\label{sum=prod of sum}
 \sum_{n\in\mathcal{N}_1\cdots\mathcal{N}_k}a_n=\prod_{i=1}^k\sum_{m\in\mathcal{N}_i}a_m.
\end{equation}
If $n\in\mathcal{N}_1\cdots\mathcal{N}_k$ decomposes as
$n=n_1\cdots n_k$, where $n_i\in\mathcal{N}_k$ for $i=1,\cdots k$, then $n_1$ and $n_2\cdots n_k$ are $\mathfrak{\Delta}$-coprime. Since  $\{a_n\}_{n\in\mathbb{N}}$ is $\mathfrak{\Delta}$-multiplicative, we have $a_n=a_{n_1}a_{n_2\cdots n_k}$, and thus by induction $a_n=a_{n_1}\cdots a_{n_k}$. This immediately gives (\ref{sum=prod of sum}).
\end{proof}

\noindent\textbf{Proof of Theorem \ref{cyclicity Ssim multiplicative}.}
We assume $F\neq0$ without  loss of generality.
It follows from the Cauchy-Schwarz inequality that for
any $\zeta\in\mathbb{D}_2^\infty$,  $\{a_n\zeta^{\alpha(n)}\}_{n\in\mathbb{N}}$ is an absolute  summable $\mathfrak{\Delta}$-multiplicative sequence. Then by Lemma \ref{summable lem}, we have
$$F(\zeta)=\sum_{n=1}^{\infty}a_n\zeta^{\alpha(n)}=\prod_{k=1}^\infty\sum_{m\in\mathcal{N}_k}a_m\zeta^{\alpha(m)}=\prod_{k=1}^\infty F_k(\zeta).$$
This implies that $\prod_{k=1}^\infty F_k$ converges pointwise to $F$. Also, the functions $F_1,F_2,\cdots$ have mutually independent variables. In fact, for the partition $\mathfrak{\Delta}=\{\mathcal{P}_1,\mathcal{P}_2,\cdots\}$, put $S_k=\{n\in\mathbb{N}:p_n\in\mathcal{P}_k\}$, where $p_n$ denotes $n$th prime number. Then $S_1,S_2,\cdots$ are pairwise disjoint, and $F_k\in H^2(\mathbb{D}^{S_k})$ for each $k\in\mathbb{N}$.
Applying Theorem \ref{prod theorem} gives the desired conclusion.
$\hfill \square $
\vskip2mm

In what follows that we will turn to another generalization  of Hartman's result. Let $S$ be a subset of $\mathcal{P}$. A pair of positive integers $m,n$ is called \textit{$S$-coprime} if for each $p\in S$, either $\gcd(m,p)=1$ or $\gcd(n,p)=1$. This  is equivalent to $\mathbf{P}(n)\cap\mathbf{P}(m)\cap S=\emptyset$.  A sequence $\{a_n\}_{n\in\mathbb{N}}$ of complex numbers is said to be \textit{$S$-multiplicative}  if $a_{mn}=a_ma_n$ for any pair $m,n$ of $S$-coprime positive integers. It is easy to see that each  $S$-multiplicative sequence is necessarily  multiplicative sequence, and  $S=\emptyset$ and $S=\mathcal{P}$ correspond to   totally  multiplicative  and multiplicative, respectively.

 \begin{thm} Let $F(\zeta)=\sum_{n=1}^\infty a_n\zeta^{\alpha(n)}$ be a function in $\mathbf{H}_\infty^2$ with $S$-multiplicative coefficients, and for each $p\in\mathcal{P}$, set $F_p(\zeta)=\sum_{k=0}^\infty a_{p^k}\zeta_p^k$,  then $F$ has the form
 $$F(\zeta)=\prod_{p\in S}F_p(\zeta)\,\prod_{q\in\mathcal{P}\setminus{S}}\, \frac{1}{1-a_q\zeta_q}.$$
 In particular, $F$ is cyclic if and only if for each $p\in S$, $F_p(z)= \sum_{k=0}^\infty a_{p^k}z^k$ is outer in $H^2(\mathbb{D})$.
\end{thm}
\begin{proof}
Since each  $S$-multiplicative sequence is necessarily  multiplicative, Theorem 4.4 implies that
$$ F(\zeta)=\prod_{p\in\mathcal{P}}F_p(\zeta).$$
When $q\notin S$, $S$-multiplicativeness implies that $a_{q^k}=a_q^k$ for $k=0,1,2, \cdots.$ From $\sum_k|a_{q^k}|^2=\sum_k|a_q|^{2k}<\infty$, $|a_q|<1$ for every $q\notin S$, and hence  $$F_q(\zeta)=\sum_{k=0}^\infty a_{q^k}\zeta_q^k=\sum_{k=0}^\infty a_{q}^k\zeta_q^k=\frac{1}{1-a_q\zeta_q}.$$
We obtain the  desired conclusion.
\end{proof}
\begin{cor}
Suppose that  $2\notin S\subseteq \mathcal{P}$, and   $F(\zeta)=\sum_{n=1}^\infty a_n\zeta^{\alpha(n)}\in\mathbf{H}_\infty^2$ has $S$-multiplicative coefficients. Set $G(\zeta)=\sum_{n=1}^\infty(-1)^n a_n\zeta^{\alpha(n)}.$ Then $G(\zeta)$ is cyclic if and only if $F(\zeta)$ is cyclic and $|a_2|\leq 1/2.$  In particular, if $F(\zeta)=\sum_{n=1}^\infty a_n\zeta^{\alpha(n)}$ has  totally multiplicative coefficients, then $G(\zeta)$ is cyclic if and only if  $|a_2|\leq 1/2.$
\end{cor}
\begin{proof}
By the assumption and  $S$-multiplicativeness, we see that $a_{2n}=a_2a_n$, $n=1,2, \cdots, $ and it follows that $F+G=2\sum_{n=1}^\infty a_{2n}\zeta^{\alpha(2n)}=2a_2\zeta_1 F.$ This gives that
$$G(\zeta)=(2a_2\zeta_1-1)F(\zeta).$$
 This equality implies the desired conclusion (also see Theorem 6.1).
 \end{proof}

  \begin{exam}
  As done in Introduction, for a given  function $\psi\in L^2(0,1)$, we extend it  to an odd $2$-periodic function on $\mathbb{R}$, still denoted by $\psi$. If $\psi(x)|_{(0,1)}=x$, then by Corollary 4.7, $-\frac{\sqrt{2}}{2}\pi \mathbf{B}U\psi=\sum_{n=1}^\infty(-1)^n\frac{1}{n}\zeta^{\alpha(n)}$ is cyclic, and therefore $\{\psi(kx)\}_{k\in\mathbb{N}}$ is complete in $L^2(0,1)$.
  If $\psi(x)=\sum_{n=1}^\infty(-1)^n\frac{1}{n^{s}}\sin n\pi x\ (\mathrm{Re}\,s>\frac{1}{2})$, then similarly, we have
  that $\{\psi(kx)\}_{k\in\mathbb{N}}$ is complete in $L^2(0,1)$ if and only if $\mathrm{Re}\, s\geq1$.
\end{exam}

\section{The cyclicity of composite functions}

Let $\psi_a\ (a\in\mathbb{D})$ denote the M\"{o}bius transformation of the unit disk $\mathbb{D}$
$$\psi_a(z)=\frac{a-z}{1-\bar{a}z},\quad z\in\mathbb{D}.$$
For $\lambda=(\lambda_1,\lambda_2,\cdots)\in\mathbb{D}^\infty$, set $\Psi_\lambda(\zeta)=(\psi_{\lambda_1}(\zeta_1),\psi_{\lambda_2}(\zeta_2),\cdots)$. It is clear that $\Psi_\lambda$ maps
$\mathbb{D}_2^\infty$ onto $\mathbb{D}_2^\infty$ if and only if $\lambda\in\mathbb{D}_2^\infty$.
An interesting observation is the following. If for some $\lambda\in\mathbb{D}_2^\infty$, the  composition  map
$$C_{\Psi_\lambda}(F)=F(\psi_{\lambda_1}(\zeta_1),\psi_{\lambda_2}(\zeta_2),\cdots),\quad F\in \mathbf{H}_\infty^2,$$
defines a bounded linear operator from $\mathbf{H}_\infty^2$ into itself, then
a function $F$ is cyclic if and only if $C_{\Psi_\lambda}(F)$ is cyclic.
In fact, since $C_{\Psi_\lambda}\circ C_{\Psi_\lambda}(F)=F$, the boundedness of $C_{\Psi_\lambda}$ naturally implies its invertibility. Therefore, there are positive constants $A$ and $B$, $A<B$ such that for any $\phi\in\mathbf{A}_\infty$, $$A\|\phi F-1\|\leq \|C_{\Psi_\lambda}(\phi)C_{\Psi_\lambda}(F)-1\|\leq B\|\phi F-1\|.$$ Since  $C_{\Psi_\lambda}$ maps $\mathbf{A}_\infty$ onto itself \cite{CG}, the above discussion shows  the consistency of the cyclicity of $F$ and $C_{\Psi_\lambda}(F)$.
%

The goal of  this section is to find appropriate conditions for a sequence $\{\phi_n:\mathbb{D}_2^\infty\rightarrow\mathbb{D}\}_{n\in\mathbb{N}}$ of holomorphic functions so that the composition operator $C_\Phi:F\mapsto F(\phi_1,\phi_2,\cdots)$
 preserves the cyclicity in $\mathbf{H}_\infty^2$ in the sense that
the sufficient and necessary condition for $C_\Phi(F)$ being  cyclic  in $\mathbf{H}_\infty^2$ is that $F$ is cyclic. Inspired by the above observation, we will consider the case in which the functions $\phi_1,\phi_2,\cdots$ have mutually independent variables.

We present the main result in this section as follows.
\begin{thm}\label{compose variable-independent inner functions} Let $F$ be a function in $\mathbf{H}_\infty^2$. If $\{\eta_n\}_{n\in\mathbb{N}}$ is a sequence of nonconstant inner functions on $\mathbb{D}_2^\infty$ with mutually independent variables and $\sum_{n=1}^{\infty}|\eta_n(0)|<\infty$,
then $F(\eta_1,\eta_2,\cdots)$ is cyclic if and only if $F$ is cyclic.
\end{thm}

Before proving Theorem \ref{compose variable-independent inner functions}, we give some explanations to the statement of the theorem.
The following two results, which will also be proved later, are established to guarantee that the composite function $F(\eta_1,\eta_2,\cdots)$ in the theorem makes sense and belongs to the space $\mathbf{H}_\infty^2$.

\begin{lem} \label{Phi map D2infty to D2infty}Let
$\{\phi_n:\mathbb{D}_2^\infty\rightarrow\mathbb{D}\}_{n\in\mathbb{N}}$ be a sequence of    holomorphic functions with mutually independent variables.
Then  $\Phi=(\phi_1,\phi_2,\cdots)$ maps $\mathbb{D}_2^\infty$ into $\mathbb{D}_2^\infty$ if and only if $\sum_{n=1}^{\infty}|\phi_n(0)|^2<\infty$.
\end{lem}

\begin{prop}\label{compose independent inner functions} Suppose that
$\{\eta_n\}_{n\in\mathbb{N}}$ is a sequence of nonconstant inner functions on $\mathbb{D}_2^\infty$ with mutually independent variables.
 Then the composition operator $C_{{\eta}}: F\mapsto F(\eta_1,\eta_2,\cdots)$ is well-defined and bounded on $\mathbf{H}_\infty^2$  if and only if $\sum_{n=1}^{\infty}|\eta_n(0)|<\infty$.
In this case, $\|C_{{\eta}}\|=\big(\prod_n\frac{1+|\eta_n(0)|}{1-|\eta_n(0)|}\big)^{1/2}$.
\end{prop}

Under the assumptions of  Theorem \ref{compose variable-independent inner functions}, Lemma \ref{Phi map D2infty to D2infty} guarantees that for $F\in \mathbf{H}_\infty^2$, the function
$F(\eta_1,\eta_2,\cdots)$ is  defined pointwise on $\mathbb{D}_2^\infty$, and Proposition \ref{compose independent inner functions} further guarantees that $F(\eta_1,\eta_2,\cdots)$ is a function in $\mathbf{H}_\infty^2$.
Also, we see that the assumption $\sum_{n=1}^{\infty}|\eta_n(0)|<\infty$ in Theorem \ref{compose variable-independent inner functions}   is natural and necessary.

This section is mainly dedicated to the proof of Theorem \ref{compose variable-independent inner functions}. We first begin with the proof of Lemma \ref{Phi map D2infty to D2infty}.

\vskip2mm

\noindent\textbf{Proof of Lemma \ref{Phi map D2infty to D2infty}. }
Assume that $\lambda=(\phi_1(0),\phi_2(0),\cdots)\in\mathbb{D}_2^\infty$.
Since $\Psi_\lambda$ maps
$\mathbb{D}_2^\infty$ onto $\mathbb{D}_2^\infty$, we only need to consider the case $\lambda=0$. Now fix some $\zeta=(\zeta_1,\zeta_2,\cdots)\in\mathbb{D}_2^\infty$ and take a sequence $\{S_n\}_{n\in\mathbb{N}}$ of  pairwise disjoint subsets of $\mathbb{N}$, such that $\phi_n\in H^2(\mathbb{D}^{S_n})$.
We also set $$r_n=\max\{|\zeta_m|:m\in S_n\},\quad n\in\mathbb{N}.$$
It is clear that each $r_n<1$, and
$$\sum_{n=1}^{\infty}r_n^2\leq\sum_{m=1}^{\infty}
|\zeta_m|^2<\infty.$$
So it suffices to prove that for any given $n\in\mathbb{N}$, one has $|\phi_n(\zeta)|\leq r_n$.
Define a function on $\mathbb{D}$ as follows, $$f(z)=\phi_n(a_1(z)\zeta_1, a_2(z)\zeta_2,\cdots),$$
where $a_m(z)=\frac{z}{r_n}$ for $m\in S_n$; $a_m(z)\equiv0$ for $m\notin S_n$.
Then $f$ is a holomorphic function on $\mathbb{D}$,
and $f(\mathbb{D})\subseteq\phi_n(\mathbb{D}_2^\infty)\subseteq\mathbb{D}$.
Since $f(0)=\phi_n(0)=0$, it follows from Schwarz's lemma that
$|\phi_n(\zeta)|=|f(r_n)|\leq r_n$.
The proof is complete.
$\hfill \square $
\vskip2mm

In what follows, we establish a key proposition, which implies the sufficiency of
Theorem \ref{compose variable-independent inner functions}.

\begin{prop}\label{compose independent bounded functions}Let $\Phi=\{\phi_n:\mathbb{D}_2^\infty\rightarrow\mathbb{D}\}_{n\in\mathbb{N}}$ be a sequence of    holomorphic functions with mutually independent variables and $\sum_{n=1}^{\infty}|\phi_n(0)|<\infty$.
 Then $$C_\Phi(F)=F(\phi_1,\phi_2,\cdots)$$ defines a bounded linear operator from $\mathbf{H}_\infty^2$ into itself with the norm at most $\big(\prod_n\frac{1+|\phi_n(0)|}{1-|\phi_n(0)|}\big)^{1/2}$.  Moreover, if $F$ is a cyclic vector in $\mathbf{H}_\infty^2$, then
 $C_\Phi(F)$ is also cyclic.
\end{prop}

To prove   Proposition \ref{compose independent bounded functions}, we need the following lemma.

\begin{lem} \label{boundedness of CPsilambda}
 $C_{\Psi_\lambda}$ defines a bounded linear operator from $\mathbf{H}_\infty^2$ into itself if and only if $\sum_{n=1}^{\infty}|\lambda_n|<\infty$.
 In this case, $\|C_{\Psi_\lambda}\|=\big(\prod_n\frac{1+|\lambda_n|}{1-|\lambda_n|}\big)^{1/2}$,  and  $C_{\Psi_\lambda}(F)$ is cyclic if and only if $F$ is cyclic.
\end{lem}
\begin{proof}
It is easy to see that  $C_{\Psi_\lambda}$ is well-defined and bounded on $\mathbf{H}_\infty^2$ if and only if there exist constants $0<c<C$ such that \begin{equation}\label{cp Cp}
                              c\|p\|\leq\|p(\psi_{\lambda_1}(\zeta_1),\psi_{\lambda_2}(\zeta_2),\cdots)\|\leq C\|p\|
                            \end{equation} for any polynomial $p$. Now we claim that
$$\|p(\psi_{\lambda_1}(\zeta_1),\psi_{\lambda_2}(\zeta_2),\cdots)\|=\frac{\|p\mathbf{K}_\lambda\|}{\|\mathbf{K}_\lambda\|}.$$
Fix a polynomial $p$ and suppose that $p$ only depends on the variables $\zeta_1,\cdots,\zeta_n$ ($n\in\mathbb{N}$).
We also put $$\lambda'=(\lambda_1,\cdots,\lambda_n,0,0,\cdots)$$ and $$\lambda''=(0,\cdots,0,\lambda_{n+1},\lambda_{n+2},\cdots).$$
Then $$\mathbf{K}_{\lambda'}(\zeta)=\frac{1}{1-\overline{\lambda_1}\zeta_1}\cdots \frac{1}{1-\overline{\lambda_n}\zeta_n},\quad \zeta=(\zeta_1,\zeta_2,\cdots)\in\mathbb{D}_2^\infty.$$
It is easy to verify that
$$F\mapsto \frac{C_{\Psi_\lambda}(F)\cdot \mathbf{K}_{\lambda'}}{\|\mathbf{K}_{\lambda'}\|}$$
defines an isometry on $H^2(\mathbb{D}^n)$.
Therefore, $$\|p(\psi_{\lambda_1}(\zeta_1),\psi_{\lambda_2}(\zeta_2),\cdots)\|=\|C_{\Psi_\lambda}(p)\|
=\frac{\|C_{\Psi_\lambda}\circ C_{\Psi_\lambda}(p)\cdot \mathbf{K}_{\lambda'}\|}{\|\mathbf{K}_{\lambda'}\|}
=\frac{\|p\mathbf{K}_{\lambda'}\|}{\|\mathbf{K}_{\lambda'}\|}.$$
On the other hand, since $p\mathbf{K}_{\lambda'}$ and $\mathbf{K}_{\lambda''}$
have mutually independent variables, we have $\|p\mathbf{K}_\lambda\|=\|p\mathbf{K}_{\lambda'}\|\|\mathbf{K}_{\lambda''}\|$, and then
$$\frac{\|p\mathbf{K}_\lambda\|}{\|\mathbf{K}_\lambda\|}=\frac{\|p\mathbf{K}_{\lambda'}\|\|\mathbf{K}_{\lambda''}\|}
{\|\mathbf{K}_{\lambda'}\|\|\mathbf{K}_{\lambda''}\|}
=\frac{\|p\mathbf{K}_{\lambda'}\|}{\|\mathbf{K}_{\lambda'}\|}.$$
This proves the claim.

By the claim and (\ref{cp Cp}), $C_{\Psi_\lambda}$ is well-defined and bounded on $\mathbf{H}_\infty^2$ if and only if
$\mathbf{K}_\lambda$ is a multiplier on $\mathbf{H}_\infty^2$.   Then the lemma immediately follows from Lemma \ref{Klambda is bounded}.
In this case we see from the claim that
$$\|C_{\Psi_\lambda}\|=\frac{\|\mathbf{K}_{\lambda}\|_\infty}{\|\mathbf{K}_{\lambda}\|}=\big(\prod_n\frac{1+|\lambda_n|}{1-|\lambda_n|}\big)^{1/2}.$$
The remaining proof is obvious.

\end{proof}

 We also need a conclusion proved in a previous paper of us \cite{DG}.
Recall that a contraction  $T$ ($\|T\|\leq1$) is said to be in the class $C_{\cdot0}$ if $T^{*k}\rightarrow0$ (as $k\rightarrow\infty$) in the strong operator topology.
If $\phi:\mathbb{D}_2^\infty\rightarrow\mathbb{D}$ is holomorphic, then the multiplication operator
$M_\phi$ is a $C_{\cdot0}$-contraction on $\mathbf{H}_\infty^2$. In fact, for any $\lambda\in\mathbb{D}_2^\infty$, one has $$
 \|M_\phi^{*k}\mathbf{K}_\lambda\|=\|\overline{\phi(\lambda)}^k\mathbf{K}_\lambda\|
=|\phi(\lambda)|^k\|\mathbf{K}_\lambda\|\rightarrow0,\quad k\rightarrow\infty.
$$
Since the linear span of reproducing kernels of $\mathbf{H}_\infty^2$ is dense, we
see that $\|M_\phi^{*k}F\|\rightarrow0\ (k\rightarrow\infty)$ also holds for any  $F\in\mathbf{H}_\infty^2$.
A pair $S, T$  of operators  is called  doubly commuting if $ST=TS$ and $ST^*=T^*S$. A sequence of operators is said to be doubly commuting if any pair of operators in it are doubly commuting. Suppose that $\mathbf{T}=(T_1,T_2,\cdots)$ is a doubly commuting sequence of $C_{\cdot0}$-contractions on some Hilbert space $\mathcal{H}$,
and set $\mathbf{T}^\alpha=T_1^{\alpha_1}\cdots T_m^{\alpha_m}$ for $\alpha=(\alpha_1,\cdots,\alpha_m,0,0,\cdots)\in\mathbb{Z}_+^{(\infty)}$ ($m\in\mathbb{N}$).
It was shown that there exists a doubly commuting sequence $\mathbf{V}=(V_1,V_2,\cdots)$ of pure isometries coextending $\mathbf{T}$, on
a larger Hilbert space $\mathcal{K}\supseteq \mathcal{H}$, that is,
$V_n^*|_{\mathcal{H}}=T_n^*$ for each $n\in\mathbb{N}$ (see \cite[Subsection 2.1]{DG}).

\vskip2mm


\noindent\textbf{Proof of Proposition \ref{compose independent bounded functions}. }
 Put $\lambda=\Phi(0)=(\phi_1(0),\phi_2(0),\cdots)$. By Lemma \ref{boundedness of CPsilambda},
$C_{\Psi_\lambda}$ is an invertible bounded linear operator on $\mathbf{H}_\infty^2$.
Then it suffices to prove the case  $\lambda=0$.
Consider the sequence $\mathbf{T}=(M_{\phi_1},M_{\phi_2},\cdots)$ of multiplication operators.
Then $\mathbf{T}$ is a doubly commuting sequence of $C_{\cdot0}$-contractions.
Now take a doubly commuting sequence $\mathbf{V}=(V_1,V_2,\cdots)$ of pure isometries  coextending $\mathbf{T}=(M_{\phi_1},M_{\phi_2},\cdots)$, on
a larger Hilbert space $\mathcal{K}\supseteq \mathbf{H}_\infty^2$. Then
$$V_n^*1=M_{\phi_n}^*1=\overline{\phi_n(0)}1=0,\quad n\in\mathbb{N}.$$
This gives that $\langle \mathbf{V}^\alpha1,\mathbf{V}^\beta1\rangle_{\mathcal{K}}=0$ for any two distinct multi-indices $\alpha,\beta\in\mathbb{Z}_+^{(\infty)}$. In fact, we may assume that $\alpha_n<\beta_n$ for some $n\in\mathbb{N}$. Then
$$V_n^{*\beta_n}\mathbf{V}^\alpha1=V_n^{*\beta_n-\alpha_n}\mathbf{V}^{\alpha-\alpha_ne_n}1
=\mathbf{V}^{\alpha-\alpha_ne_n}V_n^{*\beta_n-\alpha_n}1=0,$$
and thus $$\langle \mathbf{V}^\alpha1,\mathbf{V}^\beta1\rangle_{\mathcal{K}}=\langle V_n^{*\beta_n} \mathbf{V}^\alpha1,\mathbf{V}^{\beta-\beta_ne_n}1\rangle_{\mathcal{K}}=0,$$
where $e_n$ is the sequence which is $1$ at the  $n$-th entry and $0$ elsewhere.
Therefore, $\|p(\mathbf{V})1\|_{\mathcal{K}}=\|p\|$ for any polynomial $p$. Let $P$ denote the orthogonal projection from $\mathcal{K}$ onto $\mathbf{H}_\infty^2$. Since $\mathbf{V}$ is a coextension of $\mathbf{T}$, it is routine to check that
$p(\mathbf{T})=Pp(\mathbf{V})|_{\mathbf{H}_\infty^2}$
holds for every polynomial $p$. Then we have
$$\|p(\phi_1,\phi_2,\cdots)\|=\|p(\mathbf{T})1\|
=\|Pp(\mathbf{V})1\|\leq\|p(\mathbf{V})1\|_{\mathcal{K}}=\|p\|.$$
It follows that if $\{p_k\}_{k\in\mathbb{N}}$ is a sequence of polynomials converging to some $F\in \mathbf{H}_\infty^2$ in the norm of $\mathbf{H}_\infty^2$, then $\{p_k(\phi_1,\phi_2,\cdots)\}_{k\in\mathbb{N}}$ is a Cauchy sequence in $\mathbf{H}_\infty^2$, and hence
converges to a function $G\in \mathbf{H}_\infty^2$ with $\|G\|\leq\|F\|$.
Also, $\{p_k(\phi_1,\phi_2,\cdots)\}_{k\in\mathbb{N}}$ converges pointwise to the function $F(\phi_1,\phi_2,\cdots)$ on $\mathbb{D}_2^\infty$, forcing $G=F(\phi_1,\phi_2,\cdots)$. In conclusion,
$C_\Phi:F\mapsto F(\phi_1,\phi_2,\cdots)$ is well-defined and bounded on $\mathbf{H}_\infty^2$ with operator norm at most  $1$.
Therefore,
 for any polynomial $p$,
\begin{equation}\label{CPhi(pF-1)}
 \|C_\Phi(p)C_\Phi(F)-1\|=\|C_\Phi(p F-1)\|\leq \|p F-1\|.
\end{equation}
Note that $C_\Phi(p)$ is bounded on $\mathbb{D}_2^\infty$, and hence  the conclusion on the cyclicity follows from (\ref{CPhi(pF-1)}) and
Lemma \ref{polynomials w* dense in Hinfty}. From the above proof, we see that $\|C_{\Phi}\|\leq 1$ if $\Phi(0)=0.$ Combining this fact with  a simple  reasoning gives   the  conclusion on the norm of $C_\Phi$.
$\hfill \square $
\vskip2mm

From the proof of Proposition \ref{compose independent bounded functions}, we obtain the following corollary, which will be used in the sequel.

\begin{cor} \label{cor of Propistion composite}
 Under the hypothesis of  Proposition \ref{compose independent bounded functions}, and furthermore, we assume  that $\phi_n$ is inner and $\phi_n(0)=0$ for each $n\in\mathbb{N}$. Let $E$ denote the joint invariant subspace of $\mathbf{H}_\infty^2$
      generated by $1$ for the sequence $\mathbf{T}=(M_{\phi_1},M_{\phi_2,},\cdots)$ of multiplication operators.
Then \begin{itemize}
       \item [(1)] $\{\mathbf{T}^\alpha1:\alpha\in\mathbb{Z}_+^{(\infty)}\}$ constitutes an  orthonormal basis of $E$, and for $\alpha,\beta\in\mathbb{Z}_+^{(\infty)}$, $\mathbf{T}^{\beta*}\mathbf{T}^\alpha1\neq0$ only when $\alpha-\beta\in \mathbb{Z}_+^{(\infty)}$;
       \item [(2)] $C_\Phi$ is an isometry on $\mathbf{H}_\infty^2$;
       \item [(3)] $\mathrm{Ran}\, C_\Phi=E$, that is, to each $F\in E$, there corresponds a unique $G\in\mathbf{H}_\infty^2$ such that $F=G(\phi_1,\phi_2,\cdots)$.
     \end{itemize}
\end{cor}
\begin{proof}
Since $\phi_n$ is inner  for each $n\in\mathbb{N}$, $\mathbf{T}=(M_{\phi_1},M_{\phi_2,},\cdots)$ is a doubly commuting  sequence of pure isometries. Conclusion (1) and (2) thus immediately follow from the proof of Proposition \ref{compose independent bounded functions}. Consequently,
$\mathrm{Ran}\,C_\Phi$ coincides with
the closure of the set
$$\{C_\Phi(p):p\ \mathrm{is\ a\ polynomial}\}=\{p(\mathbf{T})1:p\ \mathrm{is\ a\ polynomial}\}.$$
This proves (3).
\end{proof}

Proposition \ref{compose independent inner functions} follows immediately from this corollary.
\vskip1.5mm
\noindent\textbf{Proof of Proposition \ref{compose independent inner functions}. }
Put $\lambda=(\eta_1(0),\eta_2(0),\cdots)$ and
$$\vartheta=(\vartheta_1,\vartheta_2,\cdots)=\Psi_\lambda(\eta_1,\eta_2,\cdots).$$ It follows from Corollary \ref{cor of Propistion composite} (2) that for any polynomial $p$,
\begin{equation*}
  \begin{split}
     \|p(\eta_1,\eta_2,\cdots)\| & =\|(p\circ\Psi_\lambda)(\Psi_\lambda(\eta_1,\eta_2,\cdots))\| \\
       & =\|(C_{\Psi_\lambda}(p))(\vartheta_1,\vartheta_2,\cdots)\| \\
       &=\|C_{\vartheta}(C_{\Psi_\lambda}(p))\|\\
       & =\|C_{\Psi_\lambda}(p)\|.
  \end{split}
\end{equation*}
This together with Lemma \ref{boundedness of CPsilambda} gives Proposition \ref{compose independent inner functions}.
$\hfill \square $
\vskip2mm

Now let $L$ denote the linear span of $$\{\mathbf{K}_\lambda:\lambda\in\mathbb{D}_2^\infty\  \mathrm{with\ finitely\ many\ nonzero\ entries}\}.$$
Then  $L$ is a subspace of $\mathbf{A}_\infty$. We also need the following.

\begin{lem} \label{L is w* dense} The subspace $L$ is uniformly dense in $\mathbf{A}_\infty$.
\end{lem}
\begin{proof} It suffices to show that for any $\alpha\in\mathbb{Z}_+^{(\infty)}$, the monomial $\zeta^\alpha$ belongs to the uniform closure of $L$.
  For this, fix a  $\alpha=(\alpha_1,\cdots,\alpha_n,0,0,\cdots)\in\mathbb{Z}_+^{(\infty)}$.
  Since $L\bigcap H^2(\mathbb{D}^n)$ is  dense in $H^2(\mathbb{D}^n)$, we can take a sequence $\{\phi_k\}_{k\in\mathbb{N}}$ in $L$ depending only on the variables $\zeta_1,\cdots,\zeta_n$ and converging to $\zeta^\alpha$ in the norm of $H^2(\mathbb{D}^n)$. In particular, $\{\phi_k\}_{k\in\mathbb{N}}$ converges uniformly to $\zeta^\alpha$ on $r\mathbb{D}^n$ for each $0<r<1$. This means that $\|(\phi_k)_r-r^{|\alpha|}\zeta^\alpha\|_\infty\rightarrow0$ as $k\rightarrow\infty$, where $|\alpha|=\alpha_1+\cdots+\alpha_n$ and $F_r$ denotes the function
  $F_r(\zeta)=F(r\zeta)$ for $F\in \mathbf{H}_\infty^2$.
  It remains to prove that $\phi_r\in L$ for any $\phi\in L$ and $0<r<1$.
  This can be deduced directly from the fact that if  $\lambda\in\mathbb{D}_2^\infty$, then $(\mathbf{K}_\lambda)_r=\mathbf{K}_{r\lambda}$.
\end{proof}

We are ready to prove Theorem \ref{compose variable-independent inner functions}.

\noindent\textbf{Proof of Theorem \ref{compose variable-independent inner functions}.}
According to Proposition \ref{compose independent bounded functions} we only need  to  verify the necessity. We suppose that $F(\eta_1,\eta_2,\cdots)$ is cyclic, and
 first assume that  $\eta_n(0)=0$ for each $n\in\mathbb{N}$. Consider the subspace $$M=\{F(\eta_1,\eta_2,\cdots)\phi:\phi\in L\}$$
 of $\mathbf{H}_\infty^2$.
By Lemma \ref{L is w* dense} and the cyclicity of $F(\eta_1,\eta_2,\cdots)$, the constant function $1$ belongs to  the norm closure of $M$. Then there exists a sequence $\{\phi_k\}_{k\in\mathbb{N}}$ in $L$ such that
$$\|F(\eta_1,\eta_2,\cdots)\phi_k-1\|\rightarrow0,\quad k\rightarrow\infty.$$

Let $E$ denote the joint invariant subspace  generated by the function $1$ for the sequence $\mathbf{V}=(M_{\eta_1},M_{\eta_2},\cdots)$ of multiplication operators, and $P$ denote the orthogonal projection from $\mathbf{H}_\infty^2$ onto $E$.
Below we will prove \begin{equation}\label{the projection of Fphi}
              P(F(\eta_1,\eta_2,\cdots)\phi_k)=F(\eta_1,\eta_2,\cdots)P\phi_k,\quad k\in\mathbb{N}.
              \end{equation}
Note that $\mathbf{V}$ is a doubly commuting sequence of pure isometries. By Corollary \ref{cor of Propistion composite} (1),  $\{\mathbf{V}^\alpha1:\alpha\in\mathbb{Z}_+^{(\infty)}\}$ constitutes an  orthonormal basis of $E$, and for $\alpha,\beta\in\mathbb{Z}_+^{(\infty)}$, $\mathbf{T}^{\beta*}\mathbf{T}^\alpha1\neq0$ only when $\alpha-\beta\in \mathbb{Z}_+^{(\infty)}$. Therefore, for any $\phi\in \mathbf{H}^\infty$, \begin{equation*}
       \begin{split}
          P(\mathbf{V}^\beta\phi)
          & =\sum_{\alpha\in\mathbb{Z}_+^{(\infty)}}\langle\mathbf{V}^\beta\phi,\mathbf{V}^\alpha1\rangle\mathbf{V}^\alpha1 \\
          & =\sum_{\alpha-\beta\in\mathbb{Z}_+^{(\infty)}}\langle\phi,\mathbf{V}^{\beta*}\mathbf{V}^\alpha1\rangle\mathbf{V}^\alpha1 \\
            & =\sum_{\gamma\in\mathbb{Z}_+^{(\infty)}}\langle\phi,\mathbf{V}^\gamma1\rangle\mathbf{V}^\beta\mathbf{V}^\gamma1 \\
            & =\mathbf{V}^\beta\biggl(\sum_{\gamma\in\mathbb{Z}_+^{(\infty)}}\langle\phi,\mathbf{V}^\gamma1\rangle\mathbf{V}^\gamma1\biggr) \\
            & =\mathbf{V}^\beta P\phi,
       \end{split}
     \end{equation*}
and thus $$P(p(\eta_1,\eta_2,\cdots)\phi)=P(p(\mathbf{V})\phi)=p(\eta_1,\eta_2,\cdots)P\phi$$ holds for any polynomial $p$. Now choose
a sequence $\{p_l\}_{l\in\mathbb{N}}$ of polynomials that converges to  $F$ in the norm of $\mathbf{H}_\infty^2$. Then
$\{P(p_l(\eta_1,\eta_2,\cdots)\phi)\}_{l\in\mathbb{N}}$ converges to  $P(F(\eta_1,\eta_2,\cdots)\phi)$ in the norm of $\mathbf{H}_\infty^2$ by Proposition \ref{compose independent bounded functions}.
On the other hand, since $\{p_l(\eta_1,\eta_2,\cdots)P\phi\}_{l\in\mathbb{N}}$ converges pointwise to $F(\eta_1,\eta_2,\cdots)P\phi$,
it follows that   $$P(F(\eta_1,\eta_2,\cdots)\phi)=F(\eta_1,\eta_2,\cdots)P\phi,$$ and hence  we have proved (\ref{the projection of Fphi}).

 It follows from  Corollary \ref{cor of Propistion composite} (3) that
for each $k\in\mathbb{N}$,
$P\phi_k=G_k(\eta_1,\eta_2,\cdots)$ for some $G_k\in \mathbf{H}_\infty^2$. By (\ref{the projection of Fphi}), we have $$P(F(\eta_1,\eta_2,\cdots)\phi_k)=F(\eta_1,\eta_2,\cdots)P\phi_k=F(\eta_1,\eta_2,\cdots)G_k(\eta_1,\eta_2,\cdots).$$
Applying  Corollary \ref{cor of Propistion composite} (2) (3), we have $FG_k\in \mathbf{H}_\infty^2$ and
\begin{equation*}
       \begin{split}
          \|FG_k-1\|
          & =\|(FG_k)(\eta_1,\eta_2,\cdots)-1\| \\
          & =\|P(F(\eta_1,\eta_2,\cdots)\phi_k-1)\| \\
            & \leq\|F(\eta_1,\eta_2,\cdots)\phi_k-1\|\rightarrow0,\quad k\rightarrow\infty.
       \end{split}
     \end{equation*}
For this end,  it suffices to prove that to any $\phi\in L$, there corresponds a function $G\in \mathbf{H}^\infty$ such that $P\phi=G(\eta_1,\eta_2,\cdots)$. For this, fix finitely many points $\xi_1,\cdots,\xi_l\ (l\in\mathbb{N})$ in $\mathbb{D}$. Put $\xi=(\xi_1,\cdots,\xi_l,0,0,\cdots)$ and $\mu=(\eta_1(\xi),\eta_2(\xi),\cdots)$. Since $\eta_1,\eta_2,\cdots$ have mutually independent variables, one can take $N\in\mathbb{N}$ sufficiently large so that  for any $n>N$, $\eta_n$ only depends on the variables $\zeta_{l+1},\zeta_{l+2},\cdots$. Then $$\eta_n(\xi)=\eta_n(\xi_1,\cdots,\xi_l,0,0,\cdots)=\eta_n(0)=0,$$ which shows that $\mu$ only has finitely many nonzero entries.
In particular, $\mathbf{K}_\mu$ is bounded.
Moreover, we have \begin{equation*}
       \begin{split}
          P\mathbf{K}_\xi
          & =\sum_{\alpha\in\mathbb{Z}_+^{(\infty)}}\langle\mathbf{K}_\xi,\mathbf{V}^\alpha1\rangle\mathbf{V}^\alpha1 \\
            & =\sum_{\alpha\in\mathbb{Z}_+^{(\infty)}}\overline{\langle\mathbf{V}^{\alpha}1,\mathbf{K}_\xi\rangle}\mathbf{V}^\alpha1 \\
            &
            =\sum_{\alpha\in\mathbb{Z}_+^{(\infty)}}\overline{\mathbf{V}^\alpha1(\xi)}\cdot\mathbf{V}^\alpha1 \\
            &=\sum_{\alpha\in\mathbb{Z}_+^{(\infty)}}\overline{\mu^\alpha}\mathbf{V}^\alpha1 \\
            & =\mathbf{K}_\mu(\eta_1,\eta_2,\cdots).
       \end{split}
     \end{equation*}
This proves the case when $\eta_n(0)=0$ for each $n\in\mathbb{N}$.

     For the general case, put $\lambda=(\eta_1(0),\eta_2(0),\cdots)$ and $(\vartheta_1,\vartheta_2,\cdots)=\Psi_\lambda(\eta_1,\eta_2,\cdots)$. It follows from the previous argument that $$F(\eta_1,\eta_2,\cdots)=(F\circ\Psi_\lambda)(\Psi_\lambda(\eta_1,\eta_2,\cdots))
     =(C_{\Psi_\lambda}(F))(\vartheta_1,\vartheta_2,\cdots)$$ is cyclic if and only if $C_{\Psi_\lambda}(F)$ is cyclic, which is further equivalent to that $F$ is cyclic by Lemma \ref{boundedness of CPsilambda}.
$\hfill \square $
\vskip2mm

The following corollary generalizes \cite[Theorem 3.4 (4)]{Ni1}.


\begin{cor} \label{F=f(eta)} Suppose $F=f(\eta)$, where $f\in H^2(\mathbb{D})$ and $\eta$ is a nonconstant inner function in $H_\infty^2$. Then $F$ is cyclic if and only if $f$ is an outer function in $H^2(\mathbb{D})$.
\end{cor}
\begin{proof} Set
$\eta'(\zeta)=\eta(\zeta_1,\zeta_3,\cdots,\zeta_{2n-1},\cdots)$
and $F'=f(\eta')$. Then $\eta'$ is also a nonconstant inner function, and by Theorem \ref{compose variable-independent inner functions}, $F'$ is cyclic if and only if $F$ is cyclic. Let $\tilde{f}$
denote the image of $f$ under the natural imbedding $i_1:H^2(\mathbb{D})\rightarrow \mathbf{H}_\infty^2$, that is,
$\tilde{f}(\zeta)=f(\zeta_1)$. It is also easy to see that $F'(\zeta)=\tilde{f}(\eta'(\zeta),\zeta_2,\zeta_4,\cdots)$. Note that
$\eta',\zeta_2,\zeta_4,\cdots$ is a sequence of nonconstant inner functions with mutually independent variables, applying Theorem \ref{compose variable-independent inner functions} and Lemma
\ref{cyclicity in H2(S)} we come to the desired conclusion.
\end{proof}

Applying Theorem \ref{compose variable-independent inner functions}
we also obtain the following.

\begin{exam}
  Suppose that $\{\eta_n\}_{n\in\mathbb{N}}$ is a sequence of  nonconstant inner functions with mutually independent variables and $\sum_{n=1}^{\infty}|\eta_n(0)|<\infty$. By Theorem \ref{compose variable-independent inner functions} and \cite[Theorem 5.9]{HLS}, we see that for any square-summable sequence $\{a_n\}_{n\in\mathbb{N}}$,
  $1+\sum_{n=1}^{\infty}a_n\eta_n$ is cyclic if and only if $1+\sum_{n=1}^{\infty}a_n\zeta_n$ is cyclic, if and only if $\sum_{n=1}^{\infty}|a_n|\leq1$.
\end{exam}

\begin{rem}
  The assumption in Theorem \ref{compose variable-independent inner functions} that each $\eta_n$ is inner cannot be dropped. In fact, for some fixed $0<r<1$, put $\phi_n(\zeta)=r^n\zeta_n\ (n\in\mathbb{N})$. From the proof of Proposition 3.8, we see that if $F\in\mathbf{H}_\infty^2$ has no zeros in $\mathbb{D}_2^\infty$, then
  $F(\phi_1,\phi_2,\cdots)\in\mathbf{A}_{R,\infty}$
for $1<R<\frac{1}{r}$, and therefore by Theorem 3.1,  $F(\phi_1,\phi_2,\cdots)$ is cyclic. On the other hand, the function $F$ can be chosen to be non-cyclic (for instance, see Remark \ref{r1}).
\end{rem}

\section{The cyclicity of $\mathbf{H}_\infty^p$-functions ($p>2$)}

In this section, we will prove


\begin{thm} \label{Hp multiply Hq} (1) If $F\in \mathbf{H}_\infty^p\ (p>2)$ and there exists $G\in \mathbf{H}_\infty^q\ (q>0)$, such that $FG$ belongs to $\mathbf{H}_\infty^2$ and  is cyclic, then $F$ is also cyclic.

\noindent(2) If $F\in \mathbf{H}^\infty$ and $G\in \mathbf{H}_\infty^2$, then $FG$ is  cyclic if and only if  both $F$ and $G$ are cyclic.
\end{thm}
It is easy to see that both  \cite[Theorem 5.7 (b)(c)]{HLS}  and \cite[Theorem 3.3 (1)(3)]{Ni1} immediately come from  this theorem. We mention that this theorem is closely related to Theorems 3 and 4 of $\S 2.1$ in \cite{Ni5}.  If both $ F$ and $\frac{1}{F}$ belong to $\mathbf{H}_\infty^2$, one  remains unknown   whether   $F$ is  cyclic or not, even
for $F$ only depending on two variables.

We record the following corollary.

\begin{cor} \label{cor of Hp multiply Hq}  Suppose that $2\leq p,q\leq\infty$  satisfy $\frac{1}{p}+\frac{1}{q}\leq\frac{1}{2}$.
 If $F\in \mathbf{H}_\infty^p$ and $G\in \mathbf{H}_\infty^q$ so that $FG$ is  cyclic, then
 both $F$ and $G$ are cyclic. In particular, if $F$ is a cyclic vector in $\mathbf{H}_\infty^2$,
 then $F^t$ is cyclic  for any $0<t<1$.
\end{cor}

We need some preparations for the proof of Theorem \ref{Hp multiply Hq}.

Suppose that a triple $p,q,r$ of positive numbers satisfies $\frac{1}{p}+\frac{1}{q}=\frac{1}{r}\leq\frac{1}{2}$.
Then for any $F\in \mathbf{H}_\infty^p$, $G\in \mathbf{H}_\infty^q$ and $n\in\mathbb{N}$, it follow from the H\"{o}lder inequality that
$$\|F_{(n)}G_{(n)}\|_{H^r(\mathbb{D}^n)}\leq\|F_{(n)}\|_{H^p(\mathbb{D}^n)}\|G_{(n)}\|_{H^q(\mathbb{D}^n)}\leq\|F\|_p\|G\|_q.$$
Taking the supreme  we see that $FG\in \mathbf{H}_\infty^r$ and hence the H\"{o}lder inequality holds for the Hardy spaces over the infinite-dimensional polydisk:
$$\|FG\|_r\leq\|F\|_p\|G\|_q.$$


Combining the H\"{o}lder inequality with the fact that the set of polynomials is dense in $\mathbf{H}_\infty^p\ (0<p<\infty)$, we have
\begin{lem} \label{FG in [F]}
  Suppose  that  $0<p,q\leq \infty$ satisfies $\frac{1}{p}+\frac{1}{q}\leq\frac{1}{2}$. If $F\in \mathbf{H}_\infty^p$ and $G\in \mathbf{H}_\infty^q$, then $FG\in[F]\cap[G]$.
\end{lem}

\noindent\textbf{Proof of Theorem \ref{Hp multiply Hq}.} (1) It suffices to show $FG\in[F]$.
Note  that $FG^{t}=F^{1-t}(FG)^t\ (0<t<1)$, here $F^{1-t}\in \mathbf{H}_\infty^{p/(1-t)}$ and $(FG)^t\in \mathbf{H}_\infty^{2/t}$. Since for each $0<t<1$, $\frac{1-t}{p}+\frac{t}{2}<\frac{1}{2}$, it follows that $FG^{t}$ belongs to $\mathbf{H}_\infty^r$ for  $r=r(t)\in(2,p)$, where $\frac{1}{r(t)}=\frac{1-t}{p}+\frac{t}{2}$, and
\begin{equation*}
  \begin{split}
     \|FG^{t}\|_2 & \leq\|FG^{t}\|_r\leq\|F^{1-t}\|_{\frac{p}{1-t}}\|(FG)^t\|_{\frac{2}{t}} \\
       & =\|F\|_p^{1-t}\|FG\|_2^t\leq(\|F\|_p+1)(\|FG\|_2+1).
  \end{split}
\end{equation*}
Then $FG^{t}\stackrel{\mathrm{w}}{\rightarrow}FG$ in $\mathbf{H}_\infty^2$ as $t$ approachs $1$ from below.
So we only need to prove that for any $0<t<1$, $FG^{t}\in[F]$.

Now we fix some $0<t<1$. As done in  the proof of \cite[Theorem 3.3 (3)]{Ni1},  there exists $2<r<p$ so that $FG^t\in \mathbf{H}_\infty^r$.
By $FG^{ts}=F^{1-s}(FG^t)^s\ (0\leq s\leq 1)$,
 the argument in the previous paragraph shows
  $FG^{ts}\in \mathbf{H}_\infty^r$ for any $0\leq s\leq1$. Take  $N\in\mathbb{N}$ sufficiently large so that $\frac{1}{r}+\frac{t}{Nq}\leq\frac{1}{2}$. From Lemma \ref{FG in [F]} and $G^{t/N}\in \mathbf{H}_\infty^{Nq/t}$,  we have  $FG^{ts+\frac{t}{N}}\in[FG^{ts}]\ (0\leq s\leq1)$,
and thus by induction,
$$FG^t=FG^{t(1-\frac{1}{N})+\frac{t}{N}}\in[FG^{t-\frac{t}{N}}]\subseteq[FG^{t-\frac{2t}{N}}]
\subseteq\cdots\subseteq[FG^{\frac{t}{N}}]\subseteq[F].$$
The proof of this part is complete.

\vskip2mm
(2) The necessity immediately follows from Lemma  \ref{FG in [F]}. For the sufficiency,
 choose a sequence $\{p_k\}_{k\in\mathbb{N}}$ of polynomials so that $\|Gp_k-1\|_2\rightarrow0$ as $k\rightarrow\infty$. Since $F$ is bounded, we also have $\|FGp_k-F\|_2\rightarrow0$ ($k\rightarrow\infty$), forcing $F\in[FG]$. This gives that $FG$ is  cyclic.
$\hfill \square $
\vskip2mm

In what follows, we give some applications of Theorem \ref{Hp multiply Hq} and Corollary \ref{cor of Hp multiply Hq}.

\begin{cor} \label{Hp multiply reproducing kernel}
If $F\in \mathbf{H}_\infty^p\ (p>2)$ and $\lambda\in\mathbb{D}_2^\infty$, then $F\mathbf{K}_\lambda$ is cyclic if and only if $F$ is  cyclic. In particular, the product of any finitely many reproducing kernels of $\mathbf{H}_\infty^2$ is cyclic.
\end{cor}
\begin{proof}  The proof is due to the fact that $\mathbf{K}_\lambda\in\bigcap_{2\leq q<\infty}\mathbf{H}_\infty^q$ and $\mathbf{K}_\lambda^{-1}\in \mathbf{H}_\infty^2$, which was proved by Nikolski (\cite[pp. 1621]{Ni1}, also see \cite{CG}). This gives
$$F\mathbf{K}_\lambda\in\bigcap_{2\leq r<p}\mathbf{H}_\infty^r.$$ Applying Theorem \ref{Hp multiply Hq} (1) we obtain the desired conclusion.
\end{proof}


\begin{exam}
  In \cite{HLS}, the space $\mathcal{H}_d^2$ of Dirichlet series $\sum_{n=1}^{\infty}a_nn^{-s}$ with $\sum_{n=1}^{\infty}|a_n|^2d(n)<\infty$ was introduced, where $d(n)$ denotes the number of divisors of a positive integer $n$.
  The image of $\mathcal{H}_d^2$ under the Bohr transform $\mathbf{B}$, denoted by $\mathcal{D}_\infty$, is exactly the  tensor product of countably infinitely many Dirichlet spaces
  $$
\mathcal{D}=\{f=\sum_{n=0}^\infty a_nz^n\ \mathrm{is\ holomorphic\ on}\ \mathbb{D}:\|f\|^2=\sum_{n=0}^\infty|a_n|^2(n+1)<\infty\}
$$
   with stabilizing sequence $\{1\}_{n\in\mathbb{N}}$ (see \cite{Be} for a definition of the  tensor product of countably infinitely many Hilbert  spaces).
  For any polynomial $p=\sum_{n=1}^{\infty}a_n\zeta^{\alpha(n)}$ ($a_n=0$ except for finitely many $n$),
  $$\|p\|_4^4=\|p^ 2\|_2^2=\sum_{n=1}^{\infty}\left|\sum_{kl=n}a_ka_l\right|^2\leq\sum_{n=1}^{\infty}\sum_{kl=n}|a_ka_l|^2d(n).$$
  Since $d(n)\leq d(k)d(l)$ for $n=kl$, we further have
  $$\|p\|_4^4\leq\sum_{n=1}^{\infty}\sum_{kl=n}(|a_k|^2d(k))(|a_l|^2d(l))=\left(\sum_{n=1}^{\infty}|a_n|^2d(n)\right)^2.$$
  Since the set of polynomials is dense in $\mathcal{D}_\infty$, one has $\mathcal{D}_\infty\subseteq \mathbf{H}_\infty^4$. By comparing the coefficients one can easily check that if $F,G\in\mathbf{H}_\infty^4$, then $FG\in\mathbf{H}_\infty^2$ and $\mathbf{B}^{-1}(F)\mathbf{B}^{-1}(G)=\mathbf{B}^{-1}(FG)$.
  This implies  that if $f,g\in\mathcal{H}_d^2$, then $fg\in\mathcal{H}^2$.
    Therefore, Corollary  \ref{cor of Hp multiply Hq}
  generalizes a result proved in \cite{HLS}: if $f,\frac{1}{f}\in\mathcal{H}_d^2$, then $f$ is a cyclic vector in $\mathcal{H}^2$ for the multiplier algebra.
\end{exam}

Finally, we apply Theorem \ref{Hp multiply Hq} to give a stronger version of a central result in \cite{No}.
    Let $S$ denote the Hardy shift $M_z$ on the Hardy space $H^2(\mathbb{D})$. Noor established some equivalent conditions for the Riemann hypothesis by constructing a semigroup of weighted composition operators $$W_nf(z)=\frac{1-z^n}{1-z}f(z^n),\quad n\in\mathbb{N}$$
on the Hardy space $H^2(\mathbb{D})$ which satisfies
\begin{equation}\label{identity in No}
 T_n(I-S)=(I-S)W_n,\quad n\in\mathbb{N},
\end{equation}
where $T_nf=f(z^n),\ f\in H^2(\mathbb{D})$ \cite{No}. Therefore, if $f$ is a cyclic vector for $\{W_n\}_{n\in\mathbb{N}}$, then $(I-S)f$ is necessarily cyclic for $\{T_n\}_{n\in\mathbb{N}}$. In the opposite direction, it is an extremely interesting problem. This will be  illustrated by the following  statement.

Note that  $H_0^2=H^2(\mathbb{D})\ominus\mathbb{C}$ is jointly reducing for $\{T_n\}_{n\in\mathbb{N}}$, and  it follows that
  if $f\in H^2(\mathbb{D})$ and $f(0)\neq0$, then $\{f(z^n)\}_{n\in\mathbb{N}}$ is complete in $H^2(\mathbb{D})$ if and only if $\{f(z^n)-f(0)\}_{n\in\mathbb{N}}$ is complete in $H_0^2$.
Recall that the Bohr transform $\mathcal{B}$ on  $H_0^2$ satisfies
$$
\mathcal{B}T_nf=M_{\zeta^{\alpha(n)}}\,\mathcal{B}f
,\quad f\in H_0^2,\,n\in\mathbb{N}.
$$
This means that if $f(0)\not=0$,  $\{f(z^n)\}_{n\in\mathbb{N}}$ is complete in $H^2(\mathbb{D})$ if and only if  $\mathcal{B}(f-f(0))$ is  cyclic in $\mathbf{H}_\infty^2$.
In particular, we have
\begin{prop}\label{varphi(zn)}
 For each fixed  $m\geq2$, set $\varphi_m=\log(1-z^m)-\log(1-z)-\log m$.
  Then $\{\varphi_m(z^n)\}_{n\in\mathbb{N}}$ is complete in $H^2(\mathbb{D})$, that is, $\varphi_m$ is a  cyclic vector for   $\{T_n\}_{n\in\mathbb{N}}$.
\end{prop}
\begin{proof}
  Put $\varphi_0=\log(1-z^m)-\log(1-z)$.
   Then by the above comment, it suffices to show that $\mathcal{B}\varphi_0$ is cyclic in  $\mathbf{H}_\infty^2$.  First note that
  $$\psi(\zeta)=\mathcal{B}(\log(1-z))(\zeta)=-\sum_{n=1}^{\infty}\frac{\zeta^{\alpha(n)}}{n}$$ is cyclic in $\mathbf{H}_\infty^2$.
      Since
   $$\mathcal{B}\varphi_0(\zeta)=\big(\mathcal{B}T_m(\log(1-z))-\mathcal{B}(\log(1-z))\big)(\zeta)=(\zeta^{\alpha(m)}-1)\psi(\zeta), $$
 Theorem 6.1(2) shows that $\mathcal{B}\varphi_0$ is cyclic.
\end{proof}

For each fixed $m\geq 2$, as done in \cite{No}, set $$h_m=\frac{1}{1-z}(\log(1-z^m)-\log(1-z)-\log m).$$
By the equalities
$$T_n\varphi_m=T_n(I-S)h_m=(I-S)W_nh_m,$$
and  Proposition \ref{varphi(zn)}, we deduce that
$(I-S)\mathrm{span}\{W_nh_m:\,n\geq1\}$ is dense in $H^2(\mathbb{D})$. This is a stronger version than  \cite[Theorem 9]{No}. However, it is extremely difficult to determine whether
$\mathrm{span}\{W_nh_m\}_{n\in\mathbb{N}}$ is dense in $H^2(\mathbb{D})$, or not, because if it was, then by \cite[Theorem 8]{No} the Riemann hypothesis would be true! In particular, here we take $m=2$, and therefore $h_2=\frac{1}{1-z}\log\frac{1+z}{2}$. If $\mathrm{span}\{W_nh_2\}_{n\in\mathbb{N}}$ is dense in $H^2(\mathbb{D})$, then the Riemann hypothesis holds, although we have known that $(I-S)\mathrm{span}\{W_nh_2\}_{n\in\mathbb{N}}$ is dense in $H^2(\mathbb{D})$.

\section{The cyclicity of functions with images outside a curve}

In this section we will use  a Riemann surface approach to reveal  the relationship between cyclicity of a function $F\in \mathbf{H}_\infty^2$ and the geometry of its image.  It is shown that if  one can find a simple curve $\gamma$ starting from the origin and tending to the infinity with bounded variation of argument, so that $\gamma$ does not intersect with the image  of    $F$, then $F$ is cyclic.

Let $X$ be a topological space. Recall that a curve in $X$ is  a continuous map $\gamma:I\rightarrow X$ on some interval $I\subseteq\mathbb{R}$.
By a simple curve $\gamma$, we mean that $\gamma$ has no self-intersection point. It is known that for every curve $\gamma:I\rightarrow\mathbb{C}^*=\mathbb{C}\setminus\{0\}$, there exists a continuous real function $\theta$ on $I$ such that $\gamma(t)=|\gamma(t)|e^{i\theta(t)}$ for any $t\in I$. Such function $\theta$ is called the argument function of $\gamma$. Our main result in this section reads as follows.

\begin{thm}\label{image outside a curve} Suppose that $F\in \mathbf{H}_\infty^2$ has no zeros in $\mathbb{D}_2^\infty$. If there exists a simple curve $\gamma:[0,1)\rightarrow\mathbb{C}$ from the origin  reaching out to infinity(that is,   $\gamma(0)=0, \lim_{t\rightarrow1^-}\gamma(t)=\infty,$) and $\gamma$ has a bounded argument function on $(0,1)$, such that the image of the function $F$ does not intersect with the  curve $\gamma$, then $F$ is cyclic.
\end{thm}

The following corollary immediately follows.
\begin{cor} Suppose that $F\in \mathbf{H}_\infty^2$ has no zeros in $\mathbb{D}_2^\infty$. The function $F$ is cyclic  if $F$ satisfies one of the following conditions:
\begin{itemize}
  \item [(1)] the image  $F(\mathbb{D}_2^\infty)$ of $F$ does not intersect with a (straight) half-line starting from  the origin;
  \item [(2)]  the image  of $F$  is the interior of  some $C^1$-smooth Jordan curve;
  \item [(3)] the image  of $F$  is a bounded  simply connected domain, and there exists an open disk $D$ whose  boundary  goes through  the origin $z=0$, such that $F(\mathbb{D}_2^\infty)\cap D=\emptyset$.
\end{itemize}
\end{cor}

In what follows that  we introduce the notion of quasi-homogeneous functions on  $\mathbb{D}_2^\infty$. For a directed weight sequence of nonnegative integers with at least a nonzero weight
$$\mathbf{K}=(K_1,K_2,\cdots),$$
we say that a function $F$ on $\mathbb{D}_2^\infty$ is $\mathbf{K}$-quasi-homogeneous, and is  of $\mathbf{K}$-degree
$m$ if for any $\xi$ with $|\xi|=1$, $F(\xi^{K_1}\zeta_1, \xi^{K_2}\zeta_2,\cdots)=\xi^{m}F(\zeta_1, \zeta_2,\cdots) $ for any $(\zeta_1, \zeta_2,\cdots)\in \mathbb{D}_2^\infty$.  If $F$ is a $\mathbf{K}$-quasi-homogeneous holomorphic function on $\mathbb{D}_2^\infty$, and is  of nonzero  $\mathbf{K}$-degree, then it is easy to verify that  the image  of $F$ is an open disk whose   center   is the origin $z=0$, and   radius equals $\sup_{\zeta\in\mathbb{D}_2^\infty}  |F(\zeta)|.$ Indeed, if $\alpha\in F(\mathbb{D}_2^\infty)$, then quasi-homogeneousness implies that $F(\mathbb{D}_2^\infty)$ contains the circle$\{z:\,|z|=|\alpha|\}.$ Let $\omega$ be such that $F(\omega)=\alpha$, then for $0\leq t\leq 1$, $F(t\omega)$ is a continuous curve from the starting point  $0$ to the end point  $\alpha$. Combining the above two facts and openness of the image  of $F$ deduces the desired conclusion.

For quasi-homogeneous functions,  the following result comes from the above conclusion and Corollary 7.2.
\begin{cor} Let $F\in \mathbf{H}_\infty^2$ be a nonzero  quasi-homogeneous  function with  nonzero  degree, $a\in \mathbb{C}$, then  $a+F$ is cyclic if and only if $a+F$ has  not zero point on $\mathbb{D}_2^\infty$ if and only if  $F\in \mathbf{H}^\infty$ and $\|F\|_\infty\leq |a|.$
\end{cor}
When $F$ is $\mathbf{1}=(1,1, \cdots)$-homogeneous function, then $F$ has the form $F(\zeta)=\sum_{n=1}^\infty a_n\zeta_n$. Corollary 7.3 implies that $a+F$ is cyclic  if and only if $\sum_{n=1}^\infty |a_n|\leq |a|$ because in this case $\sup_{\zeta\in\mathbb{D}_2^\infty}  |F(\zeta)|=\sum_{n=1}^\infty |a_n|.$ This is a previous known result.

\vskip2mm
We  will prove Theorem \ref{image outside a curve} by virtue of Theorem \ref{Imlog is bounded} below.

Let $F$ be a holomorphic function on $\mathbb{D}_2^\infty$ without zeros. Since $\mathbb{D}_2^\infty$ is simply connected,
one can define the logarithm $\log F$ of $F$ by lifting $F$ with respect to the covering  $\exp:\mathbb{C}\rightarrow\mathbb{C}^*$ (see \cite{Mun} for instance).
The argument function $\arg F$ of $F$ is defined to be the imaginary part $\mathrm{Im}\,(\log F)$ of  some branch of $\log F$.
Obviously, the argument function is unique modulo $2\pi$.

\begin{thm} \label{Imlog is bounded}
 Suppose that $F\in \mathbf{H}_\infty^2$ has no zeros in $\mathbb{D}_2^\infty$. If $F$ has a bounded argument function, then $F$ is cyclic.
\end{thm}
\begin{proof} Since $F$ has a bounded argument function $\arg F$,  one can take $N\in\mathbb{N}$ large enough so that $$\frac{|\arg F|}{N}<\frac{\pi}{2}.$$ Then there corresponds a branch of $F^{\frac{1}{N}}$ such that the arguments of values of $F^{\frac{1}{N}}$ lie in the interval $(-\frac{\pi}{2},\frac{\pi}{2})$, which forces that the real part of $F^{\frac{1}{N}}$ is always positive on $\mathbb{D}_2^\infty$. Therefore, for any $\varepsilon>0$, we have $\left|\frac{1}{F^{\frac{1}{N}}+\varepsilon}\right|<\frac{1}{\varepsilon}$,
and $$\left|\frac{F}{(F^{\frac{1}{N}}+\varepsilon)^N}\right|=\left|\frac{F^{\frac{1}{N}}}{F^{\frac{1}{N}}+\varepsilon}\right|^N\leq1$$ on $\mathbb{D}_2^\infty$.
It follows that $\frac{F}{(F^{\frac{1}{N}}+\varepsilon)^N}\ (\varepsilon>0)$ belongs to $[F]$, and it converges weakly to the constant function $1$ as $\varepsilon$ approaches $0$ from above. This completes the proof.
\end{proof}

Intuitively, the assumption that $F$ has a bounded argument function means that the image of $F$ spreads out
on connected finitely many layers of complex planes. So it is natural to represent the image of $F$ on
the Riemann surface $\mathcal{S}_{\log}=\{(z,k):z\in\mathbb{C}^*,k\in\mathbb{Z}\}$ of the logarithm. Let $\pi_{\mathbb{C}^*}$, $\pi_{\mathbb{Z}}$ denote the canonical projection from $\mathcal{S}_{\log}$ onto $\mathbb{C}^*$ and $\mathbb{Z}$, respectively. That is,
$\pi_{\mathbb{C}^*}(z,k)=z$  and $\pi_{\mathbb{Z}}(z,k)=k$ for $z\in\mathbb{C}^*$, $k\in\mathbb{Z}$.
Since
$\pi_{\mathbb{C}^*}$ is a covering,
every holomorphic function $F$ on $\mathbb{D}_2^\infty$ without zeros can be lifted to a holomorphic map $\widetilde{F}:\mathbb{D}_2^\infty\rightarrow \mathcal{S}_{\log}$ satisfying $F=\pi_{\mathbb{C}^*}\circ\widetilde{F}$.
It is routine to check that $\arg F$ is bounded if and only if $\pi_{\mathbb{Z}}(\widetilde{F}(\mathbb{D}_2^\infty))$
is a finite set. The following thus restates Theorem \ref{Imlog is bounded}.

\begin{cor} \label{RS of log} Suppose that $F\in \mathbf{H}_\infty^2$ has no zeros in $\mathbb{D}_2^\infty$. If $\pi_{\mathbb{Z}}(\widetilde{F}(\mathbb{D}_2^\infty))$ is a finite set  for a lifting  $\widetilde{F}$ of $F$ with respect to the covering $\pi_{\mathbb{C}^*}$, then $F$ is cyclic.
\end{cor}

We are ready to prove Theorem \ref{image outside a curve}.

\noindent\textbf{Proof of Theorem \ref{image outside a curve}. } 
Fix  liftings $\widetilde{F}$, $\tilde{\gamma}$ of  $F$ and $\gamma$  with respect to the covering $\pi_{\mathbb{C}^*}$, respectively. Without   loss of generality, we may assume $F(0)=1$ and  $\widetilde{F}(0)=(1,0)\in\mathcal{S}_{\log}$. Since $\gamma$ has a bounded argument function,
there exists a positive integer $N$ so that $|\pi_{\mathbb{Z}}(\tilde{\gamma})|<N$ on the interval $(0,1)$.
For $t\in(0,1)$, put $\tilde{\gamma}_+(t)=(z(t),k(t)+N)$ and $\tilde{\gamma}_-(t)=(z(t),k(t)-N)$, where $\tilde{\gamma}(t)=(z(t),k(t))$.
Then $\tilde{\gamma}_+$ and $\tilde{\gamma}_-$ are also liftings of ${\gamma}$, and do not intersect with each other.

Note that $\mathcal{S}_{\log}$ can be imbedded into the Riemann sphere $\mathbb{P}^1$ via the map $(z,k)\mapsto\log|z|+i(\mathrm{Arg}\ z+2k\pi)$, where $\mathrm{Arg}\ z$ denotes the principle argument of $z$. Furthermore $\tilde{\gamma}_+$ and $\tilde{\gamma}_-$ are mapped into two Jordan curves $\sigma_+$, $\sigma_-$ with both endpoints meeting at the  infinity $\infty$.  An application of the Jordan curve theorem gives that  $\sigma_+$ and $\sigma_-$ divide $\mathbb{P}^1$  into three connected components. One can see that the point $\widetilde{F}(0)$ belongs to the connected component of $\mathcal{S}_{\log}\setminus(\tilde{\gamma}_+\cup\tilde{\gamma}_-)$ enclosed by $\tilde{\gamma}_+$ and $\tilde{\gamma}_-$. Since the image of $\widetilde{F}$ is connected and do not intersect with neither $\tilde{\gamma}_+$ nor $\tilde{\gamma}_-$, it lies in the same connected component as $\widetilde{F}(0)$ does. Then we have $|\pi_{\mathbb{Z}}(\widetilde{F}(\mathbb{D}_2^\infty))|\leq2N$, which together with Corollary \ref{RS of log} completes the proof.
$\hfill \square $
\vskip2mm

\begin{Qes}
In Theorem \ref{image outside a curve}, can the assumption that $\gamma$ has a bounded argument function be dropped?
\end{Qes}

\section{The Kolzov completeness problem}

For $\theta\in(0,1]$, let $\varphi_\theta$ denote the odd $2$-periodic function on $\mathbb{R}$ defined by  $\varphi_\theta|_{(0,1)}=\chi_{(0,\theta)}$.
The Kolzov completeness problem is to  decide whether the dilation system $\mathcal{D}_\theta=\{\varphi_\theta(x), \varphi_\theta(2x),\cdots\}$ defined by $\varphi_\theta$ is complete in $L^2(0,1)$ for a given  $\theta\in(0,1]$(see\cite{Koz3} and \cite{Ni2, Ni3, Ni4}). Kozlov stated   some astonishing results without proofs on the completeness of the dilation system $\mathcal{D}_\theta$: $\mathcal{D}_\theta$ is complete for $\theta=1,\frac{1}{2},\frac{2}{3}$;
not for $\theta$ in a neighborhood of $\frac{1}{3}$ or $\theta=\frac{q}{p}$, where $p$ is an odd prime and $q$ is odd so that $\tan^2\frac{q\pi}{2p}<\frac{1}{p}$.
An earlier result of Akhiezer implies that in the case $\theta=1$, the corresponding system is complete \cite{Akh}. Until about 2015  Nikolski started to promote the  Kozlov problem.
He proved  all Kozlov's claims:  the corresponding system is complete for $\theta= 1, \frac{1}{2},\,\frac{2}{3}$, and incomplete for $\theta =\frac{1}{3},\, \frac{1}{4}$, and $\theta = \frac{q}{p}$  with $\sin^2\frac{\pi q}{2p}<\frac{1}{p+1}$ and in their sufficiently small  neighborhoods \cite{Ni4}.

\vskip1.5mm
 Let  $\psi(x)=\{x\}-\frac{1}{2}\ (x\in\mathbb{R})$,  where $\{x\}$ denotes the fraction part of $x$, and let  $p_1=2,p_2=3,\cdots$ be the sequence of consecutive prime numbers.
From the equality  $\psi(x)=\sum_{n=1}^{\infty}-\frac{1}{n\pi}\sin n\pi x$, we apply the Bohr transform in $\S2.1$ to $\psi$ to obtain the following
$$ \mathbf{B}U\psi=-\sum_{n=1}^{\infty}\frac{1}{\sqrt{2}n\pi}\zeta^{\alpha(n)}
=-\frac{1}{\sqrt{2}\pi}\mathbf{K}_{\mathbf{p}},$$
where $\mathbf{p}=(\frac{1}{p_1},\frac{1}{p_2},\cdots)$.
It is easy to verify that the following equalities hold.  $$\varphi_1(x)=2\psi(2x)-4\psi(x),$$
$$\varphi_{\frac{1}{2}}(x)=\psi(4x)-2\psi(x)-\psi(2x),$$
$$\varphi_{\frac{1}{3}}(x)=\psi(6x)-\psi(x)-\psi(2x)-\psi(3x),$$ and
$$\varphi_{\frac{2}{3}}(x)=\psi(3x)-3\psi(x).$$
Setting $F_\theta=\mathbf{B}U\varphi_\theta$, then  we have $$F_1=-\frac{2}{\sqrt{2}\pi}(\zeta_1-2)\mathbf{K}_{\mathbf{p}},$$
$$F_{\frac{1}{2}}=-\frac{1}{\sqrt{2}\pi}(\zeta_1^2-\zeta_1-2)\mathbf{K}_{\mathbf{p}}=(\zeta_1+1)(\zeta_1-2)\mathbf{K}_{\mathbf{p}},$$
$$F_{\frac{1}{3}}=-\frac{1}{\sqrt{2}\pi}(\zeta_1\zeta_2-\zeta_1-\zeta_2-1)\mathbf{K}_{\mathbf{p}},$$ and
$$F_{\frac{2}{3}}=-\frac{1}{\sqrt{2}\pi}(\zeta_2-3)\mathbf{K}_{\mathbf{p}}.$$
Since the completeness of $\mathcal{D}_\theta$ is equivalent to the cyclicity of  $F_{\theta}$ in $\mathbf{H}_\infty^2$, it follows immediately from Corollary \ref{Hp multiply reproducing kernel} that the dilation systems $\mathcal{D}_\theta$ are complete for $\theta=1,\frac{1}{2},\frac{2}{3}$.
It is easy to see  that the function $$F_{\frac{1}{3}}=-\frac{1}{\sqrt{2}\pi}(\zeta_1\zeta_2-\zeta_1-\zeta_2-1)\mathbf{K}_{\mathbf{p}}$$ has a zero $(-\frac{1}{2},-\frac{1}{3},0,0,\cdots)\in\mathbb{D}_2^\infty$, and therefore it  is not cyclic. This means that $\mathcal{D}_{\frac{1}{3}}$ is not  complete.

We observe when $\theta=1,\frac{1}{2}, \frac{1}{3}, \frac{2}{3}$, $F_\theta$ has the form
$$F_\theta=P\, \mathbf{K}_{\mathbf{p}},$$
where $P$ are  polynomials. Therefore,  it is directly related to the  Kozlov problem  whether $F_\theta$ has such a decomposition for a general $\theta\in(0,1].$

For a general $\theta\in(0,1],$  a direct calculation yields
\begin{equation} \label{Ftheta}
  F_{\theta}=\frac{\sqrt{2}}{\pi}\sum_{n=1}^\infty\frac{1-\cos n\theta\pi}{n}\zeta^{\alpha(n)}
\end{equation}
Also note that $\mathbf{K}_{\mathbf{p}}^{-1}=\sum_{n=1}^{\infty}\frac{\mu(n)}{n}\zeta^{\alpha(n)}$,
where $\mu(n)$ is the M\"{o}bius function (see \cite[pp. 29]{HLS}).
Using the fact that $\sum_{k|n}\mu(k)=0\ (n\geq2)$ \cite[Theorem 2.1]{Apo}, we have
\begin{equation}\label{Gtheta}
\begin{split}
   F_\theta\cdot\mathbf{K}_{\mathbf{p}}^{-1} & =\frac{\sqrt{2}}{\pi}\sum_{n=1}^\infty\frac{1}{n}\left(\sum_{k|n}\mu(k)(1-\cos \frac{n}{k}\theta\pi)\right)\zeta^{\alpha(n)} \\
     & =\frac{\sqrt{2}}{\pi}[1-\cos\theta\pi-\sum_{n=2}^\infty\frac{1}{n}\left(\sum_{k|n}\mu(k)\cos \frac{n}{k}\theta\pi\right)\zeta^{\alpha(n)}].
\end{split}
 \end{equation}
 It is clear that $$|\sum_{k|n}\mu(k)\cos \frac{n}{k}\theta\pi|\leq d(n),$$ where $d(n)$ is the number of divisors of $n$. This together with
 $d(n)=O(n^\varepsilon)$ for any $\varepsilon>0$ (see \cite[pp. 296]{Apo} for instance) implies that $G_\theta=F_\theta\cdot\mathbf{K}_{\mathbf{p}}^{-1}$ has $p$-summable coefficients for each $1<p\leq2$. Since the infinite tours $\mathbb{T}^\infty$ is the dual group of the group of finitely supported sequences of integers, it follows from the Hausdorff-Young inequality that $$G_\theta\in\bigcap_{2\leq q<\infty}\mathbf{H}_\infty^q.$$  From the identity
 $$F_\theta=G_\theta \mathbf{K}_{\mathbf{p}},$$
the function $G_\theta$ relates directly to the  Kolzov problem  since by Corollary \ref{Hp multiply reproducing kernel},  the cyclicity of $F_\theta$ and $G_\theta$ are coincident.
We have shown previously that if $\theta\in\{1,\frac{1}{2},\frac{1}{3},\frac{2}{3}\}$, then
$G_\theta$ only depends on finitely many variables, and in fact, $G_\theta$ are  actually  polynomials.
The following proposition shows that  the points  $\{1,\frac{1}{2},\frac{1}{3},\frac{2}{3}\}$  are  all possible points such that  $G_\theta$ are polynomials.

\begin{prop} \label{Gtheta infinite many variables}
If $\theta\notin\{1,\frac{1}{2},\frac{1}{3},\frac{2}{3}\}$, then  $G_{\theta}$
depends on  infinitely many variables.
In particular, $\varphi_{\theta}(x)$ cannot be expressed as a linear combination of finitely many functions in  $\{\psi(x),\psi(2x),\cdots\}$.
\end{prop}
\begin{proof} To reach a contradiction, suppose that
$G_{\theta}$ is a function in variables $\zeta_1,\cdots,\zeta_m\ (m\in\mathbb{N})$.
Let  $\mathcal{N}$ denote the subsemigroup of $\mathbb{N}$ generated by $\{1,p_{m+1},p_{m+2},\cdots\}$.
Note that
$$\mathbf{K}_{\mathbf{p}}=\prod_{n=1}^{\infty}\frac{1}{1-\frac{1}{p_n}\zeta_n}
=\prod_{i=1}^{m}\frac{1}{1-\frac{1}{p_i}\zeta_i}\cdot\prod_{j=m+1}^{\infty}\frac{1}{1-\frac{1}{p_j}\zeta_j}.$$
Then by (\ref{Ftheta}), for $n\in\mathcal{N}$,
\begin{equation*}
  \begin{split}
     \frac{\sqrt{2}(1-\cos n\theta\pi)}{n\pi} & =\langle F_{\theta},\zeta^{\alpha(n)}\rangle \\
       & =\langle G_{\theta}\cdot\mathbf{K}_{\mathbf{p}},\zeta^{\alpha(n)}\rangle \\
       & =\langle G_{\theta}\cdot\prod_{i=1}^{m}\frac{1}{1-\frac{1}{p_i}\zeta_i},1\rangle\langle \prod_{j=m+1}^{\infty}\frac{1}{1-\frac{1}{p_j}\zeta_j},\zeta^{\alpha(n)}\rangle \\
       & =\frac{G_{\theta}(0)}{n},
  \end{split}
\end{equation*}
which gives $\cos n\theta\pi=\cos \theta\pi$ for all $n\in\mathcal{N}$. Equivalently,
either $(n+1)\theta\pi$ or $(n-1)\theta\pi$ is an integer multiple of $2\pi$. In particular,
$\theta$ is rational. Put $\theta=\frac{s}{t}$ where $s,t$ are coprime positive integers.
It is clear that either $t|\frac{n+1}{2}$ or $t|\frac{n-1}{2}$ for all $n\in\mathcal{N}$. Without loss of generality, we may assume that $m\geq4$, and note each  prime factor of $t$ is not larger than $p_m$. Set $$n_1=3+\prod_{\substack{1\leq k\leq p_m \\ 3\nmid k}}k,\quad n_2=5+\prod_{\substack{1\leq k\leq p_m \\ 5\nmid k}}k.$$
Then for each $i=1,2$, we have $n_i\in\mathcal{N}$, and thus each factor of $t$ divides  either $\frac{n_i+1}{2}$ or $\frac{n_i-1}{2}$.  Since $n_2\equiv5(\mathrm{mod}\ 72)$, and for any $3\leq j\leq m$,  $n_1\equiv3(\mathrm{mod}\ 2p_j)$,
we see that
$$\frac{n_1+1}{2}\equiv2(\mathrm{mod}\ p_j), \quad \frac{n_1-1}{2}\equiv1(\mathrm{mod}\ p_j),$$ and $$\frac{n_2+1}{2}\equiv3(\mathrm{mod}\ 36), \quad
 \frac{n_2-1}{2}\equiv2(\mathrm{mod}\ 36).$$
It follows that $$p_j\nmid\frac{n_1+1}{2},\quad p_j\nmid\frac{n_1-1}{2}$$ for any $3\leq j\leq m$,
and $$4,6,9\nmid\frac{n_2+1}{2},\quad  4,6,9\nmid\frac{n_2-1}{2}.$$
This means that  each of $4,6,9, p_3,p_4,\cdots,p_m$ cannot  divide $t$, forcing $1\leq t\leq 3$. This contradicts with the assumption that $\theta\notin\{1,\frac{1}{2},\frac{1}{3},\frac{2}{3}\}$, and then  we complete the proof.
\end{proof}

\begin{rem}
  If $\theta$ is irrational, then for any prime number $p$,
  $$\sum_{k|p}\mu(k)\cos \frac{p}{k}\theta\pi=\cos p\theta\pi-\cos\theta\pi\neq0,$$
  and therefore $G_\theta$ depends on every variable by (\ref{Gtheta}).
\end{rem}

By Proposition 8.1, we conjecture that the dilation system $\mathcal{D}_\theta$ is complete only for $\theta\in\{1,\frac{1}{2},\frac{2}{3}\}$.

\vskip2mm
\noindent \textbf{Acknowledgement} This work is  partially supported by
 Natural Science Foundation of China.

\vskip3mm \noindent{Hui Dan, School of Mathematical Sciences, Fudan
University, Shanghai, 200433, China,
   E-mail:  hdan@fudan.edu.cn

\noindent Kunyu Guo, School of Mathematical Sciences, Fudan
University, Shanghai, 200433, China, E-mail: kyguo@fudan.edu.cn

\end{document}